\documentclass[12pt]{article}
\usepackage{amssymb,amsmath,amsthm,secdot,enumitem,bbm}
\usepackage[title, titletoc]{appendix}
\usepackage{authblk}
\DeclareMathOperator{\I}{\mathbbm{1}}%

\DeclareMathOperator{\Law}{Law}%

\def\E{\hskip.15ex\mathsf{E}\hskip.10ex}
\def\P{\mathsf{P}}
\def\V{\mathcal{V}}

\def\Var{\mathop{\mbox{\rm Var}}}
\def\cov{\mathop{\mbox{\rm cov}}}
\def\eps{\varepsilon}

\let\temp\phi
\let\phi\varphi
\let\varphi\temp

\newtheorem{Theorem} {Theorem}[section]
\newtheorem{Lemma}[Theorem]{Lemma}
\newtheorem{Proposition}[Theorem]{Proposition}

\theoremstyle{definition}
\theoremstyle{definition}\newtheorem{Remark}[Theorem]{Remark}
\theoremstyle{definition}\newtheorem{Definition}{Definition}[section]

\renewcommand{\theequation}{\thesection.\arabic{equation}}
\numberwithin{equation}{section}

\renewcommand{\ge}{\geqslant}
\renewcommand{\le}{\leqslant}

\newcommand{\nn}{\nonumber\\}
\newcommand{\nnn}{\nonumber}
\newcommand{\wt}{\widetilde}

\renewcommand{\d}{\partial}
\newcommand{\B}{\mathbf{B}}
\newcommand{\C}{\mathbf{C}}
\newcommand{\F}{\mathcal{F}}
\newcommand{\G}{\mathcal{G}}

\newcommand{\N}{\mathbb{N}}
\newcommand{\R}{\mathbb{R}}
\newcommand{\Z}{\mathbb{Z}}

\title{\textbf{Regularization by noise and flows of solutions for a stochastic heat equation}}

\author[*]{Oleg Butkovsky}
\author[$\ddagger$]{Leonid Mytnik}
\affil[{ }]{\small{Technion --- Israel Institute of Technology

Faculty of Industrial Engineering and Management

Haifa, 3200003, Israel}.}

\begin{document}

\maketitle
\renewcommand{\thefootnote}{*}
\footnotetext{ Email: \texttt{oleg.butkovskiy@gmail.com}. Supported in part by Israel Science Foundation Grant 1325/14, Russian Foundation for Basic Research Grant 13-01-00612-a, National Science Foundation Grant DMS--1440140, and a Technion fellowship.}
\renewcommand{\thefootnote}{$\ddagger$}
\footnotetext{Email: \texttt{leonid@technion.ac.il}.}

\begin{abstract}
Motivated by the regularization by noise phenomenon for SDEs  we prove
existence and uniqueness of the flow of solutions
for the non-Lipschitz stochastic heat  equation
\[\frac{\partial u}{\partial t}=\frac12\frac{\d^2 u}{\d z^2} +  b(u(t,z)) + \dot{W}(t,z),
\]
where $\dot W$ is a space-time white noise on $\R_+\times\R$  and $b$ is a bounded measurable function on $\R$.
As a
byproduct of our proof we also establish the so-called path--by--path uniqueness
for any  initial condition in a certain class  on the same set of probability one.
This extends recent results of Davie (2007) to the context of stochastic partial differential equations.
\end{abstract}

\section{Introduction}
This work deals with the  uniqueness theory for stochastic heat equations of the following form
\begin{align}\label{SPDE}
\frac{\d u}{\d t}&=\frac12\frac{\d^2 u}{\d z^2}+b(u(t,z))+\dot{W}(t,z),\quad t\ge0,\, z\in\mathbb{R},\\
u(0,z)&=q(z),\nonumber
\end{align}
where $\dot W$ is a Gaussian space-time white noise on $\R_+\times\R$, $b$ is a bounded Borel measurable function on $\R$, and $q$ is a Borel measurable function on
$\R$ satisfying certain growth conditions.  To be more precise we are  going to construct  the flow of
solutions to \eqref{SPDE} which is indexed by initial conditions $q$; we will  establish  uniqueness of the flow and  show that in fact the flow can be constructed in a  PDE sense on the set of full probability measure.

The equation~\eqref{SPDE} has been extensively studied in the SPDE literature. The strong existence and
uniqueness (in a probability sense) to that equation has been shown by Gy{\"o}ngy and Pardoux in
\cite{GyongyPardoux1993} for bounded $b$  and in \cite{GyongyPardoux1993b} for some locally unbounded
$b$.
Later, in~\cite{BGP94},  the results
were extended to the equations with the multiplicative noise.
Note that in the above references the equations are defined for the spatial variable $z\in [0,1]$, but the results could be easily extended to our setting of $z\in\R$.

The strong uniqueness
for~\eqref{SPDE} represents a phenomenon that is called ``regularization by noise''. This is the property that roughly speaking can be formulated as follows: deterministic equation without noise might not have uniqueness or existence property; however whenever the equation is perturbed by noise it has a unique solution, see  the related discussion in a recent book of Flandoli \cite{Fl11}.
This is the situation with \eqref{SPDE}: clearly one cannot make a general claim that equation \eqref{SPDE} {\it without noise} at the right hand side has a unique solution whenever $b$ is not Lipschitz, whereas, as we mentioned above, with the noise,
uniqueness holds for a large class of drifts $b$.
Note that whenever we say that there exists a unique {\it strong} solution to \eqref{SPDE} we mean by this
that on some filtered probability space  $(\Omega,\mathcal{F},(\mathcal{F}_t)_{t\ge0},\P)$ there exists a unique {\it adapted} strong solution to that equation. That, in fact, implies that regularization by noise phenomenon happens  in probability sense, as a regularization for It\^o-Walsh stochastic equation.

On the other hand, one can ask the question whether the regularization effect takes place in a purely PDE setting.
That is, one is interested whether it is possible to  find a set $\Omega'\subset \Omega$ of full probability such that for almost every $\omega\in \Omega'$, given the path
$$(t,z)\mapsto V(t,z,\omega)$$
equation~\eqref{SPDE}   in the integral, or so-called,  mild form (see equation~\eqref{SPDEmild} below),  has a unique solution.
Due to Flandoli's definition we will call the uniqueness of such kind the \emph{path-by-path uniqueness}, see \cite[Definition 1.5]{Fl11} and the discussion at \cite[Section~1.3.3]{Fl11}.

The problem of path-by-path uniqueness is interesting in itself. However it is closely related to another interesting question: existence and uniqueness of the flow of solutions indexed by initial conditions $q$ of the equation. To the best of our knowledge, not much is known about existence and uniqueness of  flows for SPDEs. Even if the drift and diffusion are very smooth functions, only the local flow property was established in \cite[Corollary 1.10]{HP15}. If the drift is Lipschitz and the diffusion coefficient is linear, the flow property was proved in~\cite{GZ11}; see also~\cite{cer03} for related results. Linear systems were considered earlier by Flandoli~\cite{Fl93}. We are not aware of any results in the literature concerning the case of non-Lipschitz coefficients; in the current paper we study an SPDE with a non--Lipschitz drift and an additive noise.

The question of regularization by noise for SDEs has been studied much more extensively. In particular, the following SDE has been thoroughly investigated:
\begin{eqnarray}
\label{sde1}
dX_t= b(X_t)+dB_t,
\end{eqnarray}
where $b$ is a measurable function and $B$ is a $d$-dimensional Brownian motion defined on a filtered probability space
${(\Omega,\mathcal{F},(\mathcal{F}_t)_{t\ge0}, \P)}$. First, it was derived by Zvonkin in~\cite{zvonkin74} for $d=1$,
that the above equation has a unique strong solution for a bounded measurable $b$. Then this result was generalized by Veretennikov in~\cite{ver80} for the multidimensional case, and later it was extended by Krylov and R\"ockner in~\cite{kr_rock05} for the case of
locally unbounded $b$ under some integrability condition. The flow property of solutions to~\eqref{sde1} was also established under essentially the same integrability condition, see \cite{FlFe13}, \cite{fgp10} and \cite{zhang_x11} for the case of non-constant diffusion
coefficients. Note that the definition of stochastic flow in the above references requires that  the solution  $\{X_t, t\ge 0\}$  is adapted  with respect to the filtration $\mathcal{F}_t$.
In particular, the strong uniqueness, is by definition, the  uniqueness  among the adapted solutions.
All the proofs use a Zvonkin-type transformation \cite{zvonkin74} that allows either to eliminate the ``non-regular'' drift or to make it more regular.  For the related recent interesting works on flows of SDEs see also~\cite{mnp15}, \cite{rez14}.

If one asks the path-by-path uniqueness for \eqref{sde1} then the first result in this direction has been
achieved by Davie in~\cite{Dav}, who showed it for a fixed initial condition $x$. Later the result has been generalized by Shaposhnikov in~\cite{Shap}, who established  path-by-path uniqueness of solutions simultaneously for
all initial conditions. Shaposhnikov also developed a new method that is based on the flow construction of Fedrizzi and Flandoli \cite{FlFe13}.
Recently the regularization by noise has been constructed also for equations driven by other types
of noises, e.g. L\'{e}vy noises: see Priola~\cite{Priola15}, where Shaposhnikov's method is used. We would also like to mention a paper by Catellier and Gubinelli \cite{CG} where a number of very interesting results
concerning regularization by noise and path--by--path uniqueness for ODEs were achieved.

Now if we get back to our SPDE setting, we can say outright that we do not have a luxury of having a  convenient Zvonkin-type transformation. That is why, we in a sense use the reverse argument: we first show path-by-path uniqueness together with some continuity with respect to initial conditions and based on this we show existence and uniqueness of the flow. To push the argument through we develop a new method that extends Davie's approach to the infinite-dimensional case. We believe that our method of proving existence of the flow is of independent interest.

In the next section we will present the main results of the paper.

\bigskip

\noindent \textbf{Acknowledgements}. The authors are grateful to Siva Athreya, Haya Kaspi, Andrew Lawrie, Eyal Neuman and Alexander Shaposhnikov for useful discussions. This material is partially based upon work supported by the National Science Foundation
under Grant No. DMS--1440140 while the first author was in residence at the Mathematical Sciences Research Institute (MSRI) in Berkeley, California, during the Fall 2015 semester. The first author is very grateful to MSRI for their
support and hospitality.

\section{Main results}\label{S:MR}

We study a one-dimensional stochastic heat equation on $\R$ with a drift \eqref{SPDE}. Let $(\Omega,\F,(\F_t)_{t\ge0},\P)$ be a probability space. Let $\dot W$ be a space-time white noise on this space adapted to the filtration. Let $p$ be a standard heat kernel
\begin{equation*}
p_t(z)=\frac{1}{\sqrt{2\pi t}}\exp(-z^2/2t),\quad t>0,\,z\in\R,
\end{equation*}
and $V$ be a convolution of the heat kernel $p$ with the white noise $\dot{W}(\cdot,\cdot,\omega)$,  that is
\begin{equation*}
V(s,t,z,\omega):=\int_s^{t}\int_{\R} p_{t-t'}(z-z') \,W(dt',dz'),\quad t\ge0,\,s\in[0,t],\,z\in\R.
\end{equation*}
In case $s=0$ for brevity we drop the first index and write $V(t,z,\omega):=V(0,t,z,\omega)$. Further we will frequently omit $\omega$ from the notation. Later on, in Lemma~\ref{L:GE:leon} we will show existence of  a modification of $V$ that is almost surely jointly continuous in $(s,t,z)$; with some abuse of notation this modification will be denoted by the same symbol $V$. As usual, here and in the sequel we use the convention that  $\int p_0 (x-y) f(y) dy:=f(x)$ for any measurable function $f$.

We say that a random function $u$ solves \eqref{SPDE} in the path--by--path sense, if $u(0,z)=q(z)$ and for $\P$-almost surely $\omega$ the following holds for any $t>0$, $z\in\R$
\begin{align}\label{SPDEmild}
u(t,z,\omega)=\int_{\R}p_t(z-z')q(z')\,dz'+\int_0^t\int_{\R} p_{t-t'}(z-z')b(u(t',z',\omega))\,dz'\,dt'+V(0,t,z,\omega).
\end{align}
We will also consider a stochastic heat equation that starts with the initial condition $q$ at time $s\ge0$.
\begin{align}\label{SPDEmildflow}
u(t,z,\omega)&=\int_{\R}p_{t-s}(z-z')q(z',\omega)\,dz'+\int_s^t\int_{\R} p_{t-t'}(z-z')b(u(t',z',\omega))\,dz'\,dt'\nonumber\\
&+V(s,t,z,\omega),\quad t> s,\,z\in\R,\\
u(s,z,\omega)&=q(z,\omega),\quad z\in\R\nnn.
\end{align}
Sometimes, when there is an ambiguity, we denote a solution to \eqref{SPDEmildflow} by $u_{s,q}(t,z,\omega)$, thus emphasizing the initial conditions. We see that for $s=0$ \eqref{SPDEmildflow} is just \eqref{SPDEmild}. We have to analyze $u_{s,q}$ for $s\ge0$ (rather than just at $s=0$) in order to prove the existence of the flow; see the proof of Theorem~\ref{T:stochasticflow}(a) below.

Note that the difference between the definition given above and the standard one (see, e.g., \cite[Definition~6.3]{Khosh}) is  that we do not
require  adaptiveness of the solution $u$ to the filtration generated by $\dot W$. Instead of it, for each fixed $\omega\in\Omega$ we treat equation \eqref{SPDEmildflow} separately as a deterministic PDE with a forcing term $V(s,\cdot,\cdot,\omega)$.

Let us now present the main results of the paper. First we define a class of functions that we
take as initial conditions to \eqref{SPDE}.

\begin{Definition}
\begin{enumerate}[label={\arabic*$)$}]
\item Let $\mu\ge0$. We say that a  measurable function $f\colon\R\to\R$ belongs to the class $\mathbf B(\mu)$, if
there exists a constant $C>0$ such that $|f(z)|\le C (|z|^\mu\vee1)$ for  $z\in\R$.
\item We say that a function $f\colon\R\to\R$ belongs to the class $\mathbf B(0+)$, if $f$ belongs
to the class $\mathbf B(\eps)$ for all $\eps>0$.
\end{enumerate}
\end{Definition}
For brevity, the class $\mathbf B(0)$ of measurable bounded functions on $\R$ will be denoted just by $\mathbf B$. If $f\in\B$ we define
$\|f\|_\infty:=\sup_{z\in\R} |f(z)|$.

Our first result proves existence and path--by--path uniqueness (see brief discussion on this concept in the introduction) of solutions to \eqref{SPDE} on some ``good'' set of full
probability measure simultaneously for \textbf{all} initial conditions in $\mathbf B(0+)$. Note that the class $\mathbf B(0+)$ (rather than, for example, $\mathbf B$) of initial conditions is chosen, since $u_{s,q}(t,\cdot)\in\mathbf B(0+)$ as shown below. Thus, if one starts 
equation \eqref{SPDEmildflow} from initial condition in $\mathbf B(0+)$, then at any $t\ge0$ the solution to this equation remains in the same class.

\begin{Theorem}\label{T:exandpbp}
Let $b\in\B$. There exists a set $\Omega'=\Omega'(b)\subset\Omega$ with the following properties:
\begin{enumerate}[label={\arabic*$)$}]
\item
$\P(\Omega')=1$.
\item Let $\omega\in\Omega'$. Then for any initial condition $q\in\mathbf B(0+)$ and any $s\ge0$ equation \eqref{SPDEmildflow} has a unique solution. This solution $u_{s,q}(t,\cdot)\in \mathbf B(0+)$ for any $t\ge0$.
\item Let $\omega\in\Omega'$ and $q_1,q_2\in\mathbf B(0+)$ be two initial conditions. If $q_1(z)=q_2(z)$ Lebesgue--almost everywhere in $z$, then $u_{s,q_1}(t,z,\omega)=u_{s,q_2}(t,z,\omega)$ for any $s\ge0$, $t>0$, $z\in\R$.
\end{enumerate}
\end{Theorem}
\begin{proof}[Sketch of the proof]
The proof of the theorem consists of three independent parts. First, in Section~\ref{S:prepst} we establish a number of useful regularity properties of $V$ (on a certain ``good'' set) and prove that a certain auxiliary operator is continuous.

Then in Section~\ref{S:exist} we prove  existence of a solution to \eqref{SPDEmildflow}. Let  $\C_0(\R)$ be the space of all continuous functions $\R\to\R$ vanishing at infinity equipped with the standard $\sup$-norm. Recall that it follows from Gy\"{o}ngy and Pardoux \cite{GyongyPardoux1993} that for any $s\ge0$, $q\in  \C_0(\R)$ there exists a set $\Omega_{s,q}$ of probability measure $1$ such that on $\Omega_{s,q}$ equation \eqref{SPDEmildflow} has a solution that starts with the initial condition $q$ at time $s$. More precisely, although in \cite{GyongyPardoux1993} the equation is considered on $(t,z)\in\R_+\times[0,1]$, the methods that are used in that paper work exactly in the same way for $(t,z)\in\R_+\times\R$.
Our goal is to show that this ``good'' set $\Omega_{s,q}$ can be chosen to be the same for all $s\ge0$, $q\in\mathbf{B}(0+)$.

To carry out this plan we fix a countable dense subset $\Xi$ of $\C_0(\R)$ and a countable dense subset $\Theta$ of $\R_+$. Since both $\Xi$ and $\Theta$ are countable, we see that \cite{GyongyPardoux1993} implies that there exists a set $\Omega_E$ of probability measure $1$ such that for any $\omega\in\Omega_E$, $s\in\Theta$, $q\in\Xi$ equation \eqref{SPDEmildflow} has a solution that starts with the initial condition $q$ at time $s$. Using  continuity of a certain integral operator (Lemma~\ref{L:cont_gen}), we will show that there exists a set of full measure $\Omega'\subset\Omega_E$ such that for any $\omega\in\Omega'$, $s\ge0$, $q\in\mathbf{B}(0+)$ equation \eqref{SPDEmildflow} has a solution that starts with the initial condition $q$ at time $s$.

Finally, in Section~\ref{S:uniq} we prove uniqueness of a solution to \eqref{SPDEmildflow}. The proof extensively used smoothing properties of an integral operator involving white noise (Theorem~\ref{T:gladkost} and Lemma~\ref{C:31}). We develop a new approach motivated by the ideas of Davie \cite{Dav}.
\end{proof}

The next theorem shows that there exists a unique flow of solutions to equation \eqref{SPDE} and that this flow is continuous.
We will see that this is a direct corollary of existence and path--by--path uniqueness
of solutions to \eqref{SPDE}.

\begin{Theorem}\label{T:stochasticflow}
Let $b\in\B$. Let $\Omega'\subset\Omega$ be from Theorem~\ref{T:exandpbp}.
\begin{itemize}
\item[a$)$]\textbf{$[$Existence of the flow$]$} There exists a mapping $(s, t, q, \omega)\mapsto \phi(s,t,q,\omega)$ with values in $\mathbf{B}(0+)$
defined for $0\le s\le t$, $q\in\mathbf{B}(0+)$, $\omega\in\Omega'$ such that
\begin{enumerate}[label={\arabic*$.$}]
\item
For any $s\ge0$, $q\in\mathbf{B}(0+)$, $\omega\in\Omega'$ the function $u_{s,q}(t,\cdot):=\phi(s,t,q,\omega)$ is a unique solution to \eqref{SPDEmildflow} that starts from the initial condition $q$ at time $s$;
\item On $\Omega'$ we have for $0\le r<s<t$
\begin{equation*}
\phi(r,t,q,\omega)=\phi(s,t,\phi(r,s,q,\omega),\omega);
\end{equation*}
\end{enumerate}
\item[b$)$]\textbf{$[$Continuity of the flow$]$}
Let $\phi$ be the mapping defined in Part a) of the theorem.
Let $(q_n)_{n\in\Z_+}$ be a sequence of functions from $\mathbf{B}(0+)$, that converges Lebesgue--almost everywhere  to $q\in\mathbf{B}(0+)$. Assume that there exist constants $C>0$, $\mu>0$ such that for any $n\in\Z_+$ one has
    \begin{equation*}
      |q_n(z)|\le C (|z|^\mu\vee1),\quad z\in\R.
    \end{equation*}
    Then on $\Omega'$ we have for $0\le s< t$, $z\in\R$
\begin{equation*}
\lim_{n\to\infty}\phi(s,t,q_n,\omega)(z)=\phi(s,t,q,\omega)(z).
\end{equation*}

\end{itemize}
\end{Theorem}
\begin{proof}[Proof of Theorem~\ref{T:stochasticflow}(a)]
By Theorem~\ref{T:exandpbp},
for any $\omega\in\Omega'$, $q\in\mathbf{B}(0+)$, $s\ge0$ equation \eqref{SPDEmildflow} has a unique solution $u_{s,q}$ that starts with
initial condition $q$ at time $s$.

Now for $0\le s\le t$, $q\in\mathbf{B}(0+)$, $\omega\in\Omega'$ define
\begin{equation*}
\phi(s,t,q,\omega):=u_{s,q}(t,\cdot,\omega).
\end{equation*}

Let us check that $\phi$ satisfies all the properties of the flow formulated in Theorem~\ref{T:stochasticflow}(a). The first  property is obvious. To check the second property we fix any $\omega\in\Omega'$, $0\le r<s$, $q\in\mathbf{B}(0+)$. For $t\ge s$ put
$u_1(t,\cdot):=\phi(r,t,q,\omega)$ and $u_2(t,\cdot):=\phi(s,t,\phi(r,s,q,\omega),\omega)$. Note that both
$u_1$ and $u_2$ are solutions to equation \eqref{SPDEmildflow} that starts with
initial condition $\phi(r,s,q,\omega)$ at time $s$. The initial condition $\phi(r,s,q,\omega)=u_{r,q}(s,\cdot,\omega)$ is in
$\mathbf{B}(0+)$ by Theorem~\ref{T:exandpbp}. Therefore, by Theorem~\ref{T:exandpbp} the solutions $u_1$ and $u_2$ coincide. Thus,
$$
\phi(r,t,q,\omega)=\phi(s,t,\phi(r,s,q,\omega),\omega).
$$
and $\phi$ is a flow of solutions to $\eqref{SPDEmildflow}$.
\end{proof}

The proof of  Theorem~\ref{T:stochasticflow}(b) is given in Section~\ref{S:cont}.

The next theorem describes smoothing properties of the noise $V$ that are  crucial for the proof of Theorems~\ref{T:exandpbp} and
\ref{T:stochasticflow}. We are interested in the regularity properties of the mapping
\begin{equation}\label{operupol}
(x,t,z)\mapsto \int_{0}^{t} b(V(r,z)+f(r,z)+x)\,dr,
\end{equation}
where $f$ belongs to a certain class of weighted H\"{o}lder functions with singularities defined below.

\begin{Definition}\label{D:H}
\begin{enumerate}[label={\arabic*$)$}]
\item Let $h,\gamma\in[0,1]$, $T, M,\mu\ge0$. We say that a measurable function $f\colon[0,T]\times\R\to\R$
belongs to the class $\mathbf{H}_T(h,\gamma,\mu,M)$ if
\begin{equation*}
|f(t,z)-f(s,z)|\le M |t-s|^h s^{-\gamma} (|z|^\mu\vee1),\quad 0<s<t\le T,\,\,z\in\R
\end{equation*}
and  $|f(t,z)|\le M (|z|^\mu\vee1)$ for  $z\in\R$, $t\in[0,T]$.
\item We say that a function $f\colon[0,T]\times\R\to\R$ belongs to the class $\mathbf{H}_T(h-,\gamma,0+)$ if
for any $\eps>0$ there exists $M>0$ such that $f\in\mathbf{H}_T(h-\eps,\gamma,\eps,M)$.
\end{enumerate}
\end{Definition}

If there is no ambiguity in time interval, we will frequently drop the subscript $T$ and write $\mathbf H$ instead of
$\mathbf{H}_T$.

\begin{Theorem}\label{T:gladkost}
Let $b\in\B$. There exists  a set $\Omega''=\Omega''(b)\subset\Omega$ with the following properties:
\begin{enumerate}[label={\arabic*$)$}]
\item
$\P(\Omega'')=1$;
\item Let $\omega\in\Omega''$. Then for any $0<\eps<3/4$, $T>0$, $h\in(1/2,1]$, $M>0$
there exists a constant $K_b=K_b(b,\omega,\eps, T,M, h)<\infty$
such that for any $\gamma\in[0,1]$, $\mu>0$, $f\in\mathbf H_T(h,\gamma,\mu,M)$, $x,y,z\in\R$, $0\le t_1\le t_2\le T$, $s\in[0,T]$ we have
\begin{align}\label{predmainest}
\bigl|\int_{t_1}^{t_2} &(b(V(t+s,z)+f(t,z)+x)-b(V(t+s,z)+f(t,z)+y))\,dt\bigr|\nnn\\
&\le K_b(\omega)|x-y|(t_2-t_1)^{1-1/4(\frac{\gamma}{h-1/4}\vee1)-\eps}(|x|\vee |y|\vee 1)^{1+\eps}(|z|^{3\mu+\eps}\vee 1) .
\end{align}
Furthermore, $\E K_b\le \|b\|_{\infty} C (\eps, T,M,h)$ for some function $C$ that does not depend on $b$.
\end{enumerate}
\end{Theorem}

If the function $b$ were a Lipschitz function,
then the left--hand side of \eqref{predmainest} would be bounded by $|t_1-t_2||x-y|$. In our case, when $b$ is just a bounded function, the left--hand side of \eqref{predmainest} is obviously bounded by $|t_1-t_2|$. Theorem~\ref{T:gladkost} implies that one can trade the regularity in $t$ to gain the regularity in $x$. In particular, we see from the above theorem that $\P$--almost surely the function
\begin{equation*}
x\mapsto\int_{t_1}^{t_2} b(V(t,z)+x)\,dt,\quad x\in\R,
\end{equation*}
is Lipschitz in $x$. Moreover, we have very good local control on coefficients.

\begin{proof}[Sketch of the proof of Theorem~\ref{T:gladkost}]
The proof is based on an application of a suitable version of Kolmogorov continuity theorem to
a corresponding moment bound. This is done in Section~\ref{S:3}. The calculation of the moment bound turned out
to be rather complicated and it involves a number of technical estimates. 
We do it thoroughly in Section~\ref{proofprop} utilizing some ideas from
\cite{CG}.
\end{proof}

\begin{Remark}
We would like to note that whilst the good set $\Omega'$ in Theorems~\ref{T:exandpbp}--\ref{T:gladkost} can be chosen independently
of the initial condition $u_0$, $\Omega'$ might still depend on the drift function $b$.
\end{Remark}

It is interesting to compare smoothing properties of operator \eqref{operupol} to the smoothing properties of a similar operator
with a Brownian motion $B$ in place of $V$, see \cite[Corollary 2.2]{Fl11} and also \cite[Lemmas 3.1 and 3.2]{Dav}. We see that since $V$ in the time variable is less regular than the Brownian motion, Theorem~\ref{T:gladkost} guarantees a better smoothing.


The function $f$ appears in \eqref{operupol} due to the presence of the drift in our main equation \eqref{SPDEmildflow}. Note that in the original Davie's paper \cite{Dav} the smoothing is considered without the drift term (this corresponds to the case $f\equiv0$). That was possible due to the use of the Girsanov transformation for eliminating the drift. In other words, the ``good'' set $\Omega''$ in  \cite{Dav} depends on the drift $f$ and the initial condition. Since we are aimed at establishing the flow property for  \eqref{SPDEmildflow} we have to prove path--by--path uniqueness simultaneously for all initial conditions, see the proof of Theorem~\ref{T:stochasticflow}(a). Thus, we have to prove that smoothing in Theorem~\ref{T:gladkost} occurs simultaneously for all drifts $f$ and this cannot be achieved with Girsanov's transformation.

\medskip
The rest of the paper is devoted to the proofs of the main results and is organized as follows. In Section~\ref{S:3}
we prove Theorem~\ref{T:gladkost}. The proof of Theorem~\ref{T:exandpbp} is rather large and is split into two parts for the convenience
of the reader. Namely, in Section~\ref{S:prepst} we establish a number of auxiliary lemmas and present the main part of the proof in Section~\ref{S:5}. Theorem~\ref{T:stochasticflow} is also proved in Section~\ref{S:5}. Section~\ref{proofKol} contains a proof of a global version of the Kolmogorov continuity theorem that is used in the paper. An important moment bound that is exploited for the proof of Theorem~\ref{T:gladkost} is derived in Section~\ref{proofprop}. A technical lemma that is applied to prove smoothing properties of the noise is
established in Section~\ref{S:PCL}. Finally, a number of technical estimates concerning the Gaussian kernel and related functions are obtained in the Appendix.

\smallskip
\textbf{Convention on constants.} Throughout the paper, $C$ denotes a positive constant whose value may change from line to line. $K$ denotes a random constant whose value might depend on $\omega\in\Omega$.

\section{Proof of Theorem \ref{T:gladkost}}\label{S:3}

We start proving the main results by presenting a proof of Theorem \ref{T:gladkost}.  First we give here a version of the Kolmogorov theorem on a noncompact set (global version) that will be extensively used in this and other proofs in the paper.

Define for $w=(w_1,w_2)\in\R^2$, $a=(a_1,a_2)\in (0,1]^2$ a weighted norm $d_a$
\begin{equation}\label{distdef}
d_a(w):=w_1^{a_1}+w_2^{a_2}.
\end{equation}

\begin{Lemma}[Kolmogorov Continuity Theorem]\label{L:VietR}
Let $X(x,y)$, $x\in\R$, $y\in\R^{2}$, be a continuous random field with values in $\R$.
Assume that there
exist nonnegative constants $a=(a_1,a_2)\in(0,1]^2$, $\alpha$, $\beta_1$, $\beta_2$,  $C$ such that the inequalities
\begin{align}
&\E |X(x_1,y_1)-X(x_1,y_2)-X(x_2,y_1)+X(x_2,y_2)|^{\alpha}\le C  |x_1-x_2|^{\beta_1}d_a(y_1-y_2)^{\beta_2},\nnn\\
&\E |X(x_1,y_1)-X(x_2,y_1)|^{\alpha}\le C  |x_1-x_2|^{\beta_1} \label{secassum}
\end{align}
hold for any $x_1,x_2\in\R$, $y_1,y_2\in\R^{2}$, $|x_1-x_2|\le 1$, $|y_1-y_2|\le 1$.

Then for any $\gamma_1\in(0,(\beta_1-1)/\alpha)$ and $\gamma_2\in(0,(\beta_2-1/a_1-1/a_2)/\alpha)$ there exist a set $\Omega^*\subset\Omega$ with $\P(\Omega^*)=1$ and a random variable $K$  with $\E K(\omega)^{\alpha}\le C_1$ such that for any $\omega\in\Omega'$, $x_1,x_2\in\R$, $y_1,y_2\in\R^2$, $|x_1-x_2|\le 1$, $|y_1-y_2|\le 1$
\begin{equation}\label{KolmogR}
|X(x_1,y_1)-X(x_1,y_2)-X(x_2,y_1)+X(x_2,y_2)\!|\le\! K(\omega) (|x_1|\vee|y_1| \vee 1)^{3/\alpha} |x_1-x_2|^{\gamma_1}d_a(y_1-y_2)^{\gamma_2},
\end{equation}
and
\begin{equation}\label{KolmogR1term}
|X(x_1,y_1)-X(x_2,y_1)|\le K(\omega) (|x_1|\vee|y_1| \vee 1)^{3/\alpha} |x_1-x_2|^{\gamma_1},
\end{equation}
where the constant $C_1>0$ depends on the field $X$ only though $a$, $\alpha$, $\beta_i$, $\gamma_i$, $C$.
\end{Lemma}

The proof of the theorem is given in Section~\ref{proofKol}.

The proof of Theorem~2.3 is based on the above mentioned version of the Kolmogorov theorem and the following moment bound.
\begin{Proposition}\label{P:MB}
Let $b\colon\R\to\R$ be a bounded differentiable function with bounded derivative. Then for any $0\le t_1\le t_2\le T$, $z,z_1, z_2,x,y\in\R$, $|z_1-z_2|\le 1$, $\delta\in(0,1)$, $\delta'\in(0,\delta)$, $p>1$ we have
\begin{align}
\E\Bigl(\int_{t_1}^{t_2} &\bigl(b'(V(t,z_1)+x)-b'(V(t,z_2)+y)\bigr)\,dt\Bigr)^p\nnn\\
&\le   C (t_2-t_1)^{p(3/4-\delta/4)}(|z_1-z_2|^{p\delta'/2}+|x-y|^{p\delta});\label{mom1}\\
\E\Bigl(\int_{t_1}^{t_2} &b'(V(t,z))\,dt\Bigr)^p\le   C(t_2-t_1)^{p(3/4-\delta)}.\label{mom2}
\end{align}
for some constant $C=C(p,T,\delta,\delta',\|b\|_{\infty})>0$.
\end{Proposition}

It is important to stress  that the constant $C$ from Proposition~\ref{P:MB} depends only on $\|b\|_{\infty}$ but not on the function $b$ itself and not on its derivative. The proof of Proposition~\ref{P:MB} is postponed to Section~\ref{proofprop}.

\smallskip
Finally we need a technical estimate.
\begin{Lemma}\label{L:prob0}
Let $U\subset \R$ and assume that $U$ has a Lebesgue measure $0$. Then there exists a set $\Omega(U)\subset\Omega$ such that
$\P(\Omega(U))=1$ and for any $\omega\in\Omega(U)$, $h>1/2$, $T>0$, $M>0$, $\mu>0$, $f\in \mathbf H_T(h,1,\mu,M)$, $z\in\R$, $s\in[0,T]$
we have
\begin{equation*}
\int_{0}^{T} \I_U(V(t+s,z,\omega)+f(t,z))\,dt=0.
\end{equation*}
\end{Lemma}

This lemma is proved in Section~\ref{S:PCL}.

\begin{proof}[Proof of the Theorem~\ref{T:gladkost}]
First we consider the case when $b$ is a bounded differentiable function with a continuous bounded derivative and $\|b\|_{\infty}= 1$. In this case, we apply a version of the Kolmogorov continuity theorem (Lemma~\ref{L:VietR}) to the random field
$$X(t,(z,x)):=\int_0^t  b'(V(r,z)+x)\,dr.
$$
Fix arbitrary $T>0$. It follows from Proposition~\ref{P:MB} and Lemma~\ref{L:VietR} that for any $\delta\in(0,1)$, $\eps>0$ there exist a set $\Omega_{T,\delta,\eps}\subset\Omega$ with $\P(\Omega_{T,\delta,\eps})=1$ and a random variable $K(\omega)$
such that for all $z_1,z_2,x,y\in\R$ with $|z_1-z_2|+|x-y|\le 1$ and $0\le t_1\le t_2\le 2T$, $\omega\in\Omega_{T,\delta,\eps}$ one has
\begin{align*}
\Bigl|&\int_{t_1}^{t_2}\bigl(b'(V(t,z_1)+x)-b'(V(t,z_2)+y)\bigr)\,dt\Bigr|\nonumber\\
&\le K(\omega)(|x|^{\eps}\vee |y|^{\eps}\vee 1)(|z_1|^{\eps}\vee |z_2|^{\eps}\vee 1)(t_2-t_1)^{(3/4-\delta/4-\eps)}(|z_1-z_2|^{\delta/2}+|x-y|^{\delta})^{1-\eps}
\end{align*}
and $\E K\le C$, where the constant $C=C(T,\delta,\eps)$ does not depend on the function $b$.

We apply now the above inequality to $z_1=z_2=z$ and arbitrary $x,y\in\R$. That is, if $|x-y|\le N$ we
apply the above inequality $N$ times. We get  that on  $\Omega_{T,\delta,\eps}$ for all $x,y,z\in\R$, $0\le t_1\le t_2\le T$, $s\in[0,T]$
\begin{align}\label{finfi}
\Bigl|\int_{t_1}^{t_2} &\bigl(b'(V(t+s,z)+x)-b'(V(t+s,z)+y)\bigr)\,dt\Bigr|\nonumber\\
&\le K_1(\omega)(|x|^{1+\eps}\vee |y|^{1+\eps}\vee 1)(|z|^\eps\vee 1)(t_2-t_1)^{3/4-\delta/4-\eps}|x-y|^{\delta(1-\eps)},
\end{align}
where we have also applied change of variables $t\rightarrow t+s$ in the integral. Here $\E K_1\le C=C(T,\delta,\eps)$.

Similarly, inequality \eqref{KolmogR1term} and Proposition~\ref{P:MB} yield that for any $\eps\in(0,3/4)$ there
exists a set $\wt\Omega_{T,\eps}$ and a random variable $K_2(\omega)$ such that for any $\omega\in\wt\Omega_{T,\eps}$, $z\in\R$, $0\le t_1\le t_2\le T$, $s\in[0,T]$ we have
\begin{equation}\label{finfi11}
\bigl|\int_{t_1}^{t_2} b'(V(t+s,z))\,dt\bigr|\le K_2(\omega)(|z|^\eps\vee 1)(t_2-t_1)^{3/4-\eps}.
\end{equation}
Again, $\E K_2\le C=C(T,\eps)$.

This allows us to proceed to the next step. Fix $M>0$ and take any $f\in\mathbf H(h,\gamma,\mu,M)$. Fix $0\le t_1\le t_2\le T$ and let $f_n$ be the following piecewise-constant approximation of $f$:
\begin{equation*}
f_n(t,z)\!:=\!\sum_{i=0}^{2^n-1}\! \I_{(t_1+(t_2-t_1)i2^{-n},t_1+(t_2-t_1)(i+1)2^{-n}]}(t) f(t_1+(t_2-t_1)(i+1)2^{-n}\!,z),\,
\, t\in[t_1,t_2], z\!\in\!\R.
\end{equation*}
Clearly, the sequence of functions $f_n$ converges pointwise to $f$.
Thus, for arbitrary ${s\in[0,T]}$, $\omega\in\Omega$, $x,y,z\in\R$ we derive
\begin{align}\label{summi}
\Bigl|\int_{t_1}^{t_2} &\bigl(b(V(t+s,z)+f(t,z)+x)-b(V(t+s,z)+f(t,z)+y)\bigr)\,dt\Bigr|\nn
=&\Bigl|\int_{x}^y \int_{t_1}^{t_2} b'(V(t+s,z)+f(t,z)+r)\,dtdr\Bigr|\nn
\le&
\Bigl|\int_{x}^y \int_{t_1}^{t_2} b'(V(t+s,z)+f_0(t,z)+r)\,dtdr\Bigr|\nn
&+\sum_{k=0}^\infty \Bigl|\int_{x}^y \int_{t_1}^{t_2} \bigl(b'(V(t+s,z)+f_{k+1}(t,z)+r)-b'(V(t+s,z)+f_k(t,z)+r)\bigr)\,dtdr\Bigr|
\nn
=&:I_1+I_2,
\end{align}
where in the first identity we have used Fubini's theorem (recall that we have assumed boundedness of the function $b'$).  It turns out that we need to apply \eqref{finfi} with different $\delta$ to estimate $I_1$ and $I_2$.
Since, by definition, the function $f_0$ is constant on $[t_1,t_2]$, we see that \eqref{finfi} with $\delta=\eps$
together with \eqref{finfi11} yield for $\omega\in \wt\Omega_{T,\eps}\cap \Omega_{T,\eps,\eps}$
\begin{align}\label{I11}
|I_1|\le& \Bigl|\int_{x}^y \int_{t_1}^{t_2} b'(V(t+s,z)+f_0(t,z)+r)-b'(V(t+s,z))\,dtdr\Bigr|\nnn\\
&+\Bigl|\int_{x}^y \int_{t_1}^{t_2} b'(V(t+s,z))\,dtdr\Bigr|\nnn\\
\le&K_3(\omega)(|x|\vee |y|\vee 1)^{1+2\eps}(|z|^{2\mu+\eps}\vee 1)(t_2-t_1)^{3/4-2\eps}|x-y|
\end{align}
and $\E K_3\le C(T,\eps,M)$.

To estimate $I_2$ we observe that each function $f_k$ is a piecewise constant function in $t$. Thus, to estimate $k$-th term of the sum we split the integral over $[t_1,t_2]$ into integrals over intervals $(t_1+(t_2-t_1)i2^{-n},t_1+(t_2-t_1)(i+1/2)2^{-n}]$, $i=0,\hdots 2^{k}-1$, where $f_k$ and $f_{k+1}$ are constant in $t$, and apply estimate \eqref{finfi} to each of these integrals. Note that $f_k=f_{k+1}$ on the complement of the union of these intervals. We obtain on  $\Omega_{T,\delta,\eps}$
\begin{align}\label{I22}
|I_2|\!\le&  K_4(\omega) |x-y|(|x|\vee |y|\vee 1)^{1+\eps}(|z|^{2\mu+\eps}\vee 1)(t_2-t_1)^{3/4-\delta/4-\eps}\nn
&\times\sum_{k=0}^\infty \sum_{i=0}^{2^k-1} 2^{-(3/4-\delta/4-\eps) k}
\bigl|f(t_1+\frac{(i+1)(t_2-t_1)}{2^{k}},z)-\!f(t_1+\frac{(i+1/2)(t_2-t_1)}{2^{k}},z)\bigr|^{\delta-\eps}\nn
\le& K_5(\omega) |x\!-\!y|(|x|\!\vee\! |y|\!\vee 1)^{1+\eps}(|z|^{3\mu+\eps}\vee 1) (t_2\!-\!t_1)^{3/4+\delta(h-1/4-\gamma)-2\eps}\sum_{k=0}^\infty 2^{k(1/4-\delta(h-1/4)+2\eps)}
\end{align}
Again, $\E K_i\le C(T,\delta,\eps,M)$ for $i=4,5$.

We see that in order for the sum to be convergent we must necessarily have
\begin{equation*}
\delta(h-1/4)>1/4+2\eps.
\end{equation*}
We must also have $\delta<1$. Thus, for $h>1/2$ one can take $\delta:=1/(4h-1)+16\eps$. If, additionally, $\gamma\le h-1/4$, then combining \eqref{summi}, \eqref{I11}, \eqref{I22} we finally obtain for $\eps>0$  on $\Omega^*_{T,\eps,h}:=\wt\Omega_{T,\eps}\cap \Omega_{T,\eps,\eps}\cap\Omega_{T,1/(4h-1)+16\eps,\eps}$
\begin{multline*}
\Bigl|\int_{t_1}^{t_2}\bigl(b(V(t+s,z)+f(t,z)+x)-b(V(t+s,z)+f(t,z)+y)\bigr)\,dt\Bigr|\\
\le
K_6(\omega)|x-y|(|x|\vee |y|\vee 1)^{1+2\eps}(|z|^{3\mu+\eps}\vee 1) (t_2-t_1)^{3/4-2\eps}.
\end{multline*}
In case $\gamma>h-1/4$, we obtain on $\Omega^*_{T,\eps,h}$
\begin{multline*}
\Bigl|\int_{t_1}^{t_2}\bigl(b(V(t+s,z)+f(t,z)+x)-b(V(t+s,z)+f(t,z)+ y)\bigr)\,dt\Bigr|\\
\le
K_6(\omega)|x- y|(|x|\vee |y|\vee 1)^{1+2\eps}(|z|^{3\mu+\eps}\vee 1) (t_2-t_1)^{1-\gamma/(4h-1)-50\eps}.
\end{multline*}

Note that in both cases, $\E K_6(\omega)\le C(T,\eps,M,h)$. Now we take $\Omega^*:=\cap \Omega^*_{T,\eps,h}$ where the intersection is among
all rational $T>0$, $\eps>0$, $h>1/2$. We see that on $\Omega^*$ the statement of the theorem holds. This concludes the proof of the theorem for the case where $b$ is a bounded differentiable function with a continuous bounded derivative and $\|b\|_{\infty}=1$.

If $\|b\|_{\infty}=0$, then there is nothing to prove. If $b$ is a bounded differentiable function with a continuous bounded derivative but $\|b\|_{\infty}\neq1$, $\|b\|_{\infty}>0$, then the statement of the theorem also holds. Indeed, we can renormalize $b$ and consider $b_1(x):=b(x)/\|b\|_{\infty}$.

Finally, to prove the theorem in the general case (for bounded but not necessarily differentiable $b$ with arbitrarily $\|b\|_{\infty}$) we
use approximations. It follows from Lusin's theorem that there exists a sequence $(b_n)_{n\in\Z_+}$ of bounded differentiable function with continuous bounded derivatives such that
\begin{equation*}
\lim_{n\to\infty}b_n(x)=b(x)\text{\,\,Lebesgue-almost everywhere in $x$;\,\,$x\in\R$}
\end{equation*}
and $\sup_n \|b_n\|_\infty\le 2 \|b\|_\infty$. Put $U:=\{x\in\R\colon \lim_{n\to\infty}b_n(x)\neq b(x)\}$, $\wt b(x)\!:=\!\lim_{n\to\infty}b_n(x)$. We see that the set $U$ is of Lebesgue measure $0$.

Let $\Omega_n$ be the ``good'' set for the function $b_n$ (i.e., a set such that the statement of the theorem is satisfied for the function $b_n$). By above, $\P(\Omega_n=1)$. Take
\begin{equation*}
\Omega_\infty:=\bigcap_{n=1}^{\infty} \Omega_n\cap \Omega(U),
\end{equation*}
where the set $\Omega(U)$ is defined in Lemma~\ref{L:prob0}. Clearly, $\P(\Omega_\infty)=1$. Take arbitrary $T>0$, $M>0$, $h>1/2$, $0<\eps<3/4$. By the dominated convergence theorem and Lemma~\ref{L:prob0} we have on $\Omega_\infty$
for any $\gamma\in[0,1]$, $\mu>0$, $f\in\mathbf H(h,\gamma,\mu,M)$, $x,y,z\in\R$, $0\le t_1\le t_2\le T$, $s\in[0,T]$
\begin{align}\label{almostend}
\bigl|\int_{t_1}^{t_2}\bigl(b&(V(t+s,z)+f(t,z)+x)-b(V(t+s,z)+f(t,z)+ y)\bigr)\,dt\bigr|\nnn\\
\le& \bigl|\int_{t_1}^{t_2}\bigl(\wt b(V(t+s,z)+f(t,z)+x)-\wt b(V(t+s,z)+f(t,z)+ y)\bigr)\,dt\bigr|\nnn\\
&+3\int_{t_1}^{t_2}\bigl(\I_U(V(t+s,z)+f(t,z)+x)+\I_U (V(t+s,z)+f(t,z)+ y)\bigr)\,dt
\nnn\\
=&\liminf_{n\to\infty}\bigl|\int_{t_1}^{t_2}\bigl(b_n(V(t+s,z)+f(t,z)+x)-b_n(V(t+s,z)+f(t,z)+ y)\bigr)\,dt\bigr|\nnn\\
\le& |x- y|(|x|\vee |y|\vee 1)^{1+\eps}(|z|^{3\mu+\eps}\vee 1) (t_2-t_1)^{1-1/4(\frac{\gamma}{h-1/4}\vee1)-\eps} \liminf_{n\to\infty} K_{b_n}(\omega),
\end{align}
where $K_{b_n}$ is the corresponding constant from Theorem~\ref{T:gladkost}. For $\omega\in\Omega_\infty$ put $ K_b(\omega):=\liminf_{n\to\infty} K_{b_n}(\omega)$. By Fatou's lemma,
\begin{equation*}
\E  K_b(\omega)\le \liminf_{n\to\infty} \E K_{b_n}(\omega)\le C (\eps,T,M,h)\sup_n \|b_n\|_\infty .
\end{equation*}
Thus the random variable $ K_b(\omega)$ has a finite expectation. Hence there exists a set $\Omega''\subset\Omega_\infty$ such that $\P(\Omega'')=1$ and on $\Omega''$ we have $ K_b(\omega)<\infty$. This together with \eqref{almostend} concludes the proof of
the theorem.
\end{proof}

\section{Preparation steps for proving Theorem~\ref{T:exandpbp} }\label{S:prepst}

In this section we prepare for the proof of our main result, that is, Theorem~\ref{T:exandpbp}. In particular, we will select
a specific ``good'' set $\Omega'$ of full probability measure and in the next section we will prove that equation
\eqref{SPDEmildflow} indeed has a unique solution on $\Omega'$.

First we need to introduce approximation operator in the following way. Let $f:[0,1]\to\R$ be a continuous function. We define a piecewise-constant  approximation of $f$ as follows. For $n\in\Z_+$ put
\begin{equation}\label{lambda}
\lambda_{n}(f)(t):=\sum_{i=0}^{2^n-1} \I_{(i2^{-n},(i+1)2^{-n}]}(t) f((i+1)2^{-n}).
\end{equation}
In other words, $\lambda_{n}(f)$ is a piecewise-constant function that takes constant values on intervals of length $2^{-n}$.

\smallskip

Many times in the proofs of the theorems it will be convenient to work with a shifted solution to \eqref{SPDEmildflow}. Thus, we define
\begin{equation}\label{ushifted}
u^*_{s,q}(t,\cdot):=u_{s,q}(t+s,\cdot),\quad t\ge0,
\end{equation}
where we recall that $u_{s,q}$ stands for the solution to \eqref{SPDEmildflow} that starts from the initial
condition $q$ at times $s$. It is easy to see that $u^*_{s,q}$ satisfies the following equation
\begin{align}\label{shiftSPDEmildflow}
u^*_{s,q}(t,z,\omega)=& \int_{\R}p_{t}(z-z')q(z',\omega)\,dz'+\int_0^t\int_{\R} p_{t-t'}(z-z')b(u^*_{s,q}(t',z',\omega))\,dz'\,dt'\nonumber\\
&+V(s,t+s,z,\omega),\quad t> 0,\,z\in\R,\\
u^*_{s,q}(0,z,\omega)=&{} q(z,\omega),\quad z\in\R.\nnn
\end{align}
\begin{Remark}\label{R:CaptObvious}
Clearly, equation~\eqref{SPDEmildflow} has a unique solution if and only if equation~\eqref{shiftSPDEmildflow} has a unique solution.
\end{Remark}

We introduce also the notation for the difference between two kernels by setting
\begin{equation}\label{DeltaP}
\Delta p_t(z_1,z_2):=p_t(z_1)-p_t(z_2),\quad t>0,\,z_1,z_2\in\R.
\end{equation}

Further, we will need to consider weighted norms. So for $\delta\ge0$ we define
weight function
\begin{equation}\label{LambdaL}
\Lambda_\delta(x):=e^{x}x^\delta,\quad x\ge0.
\end{equation}

Consider also a class of globally Lipschitz functions.
\begin{Definition}\label{D:BL}
We say that a function $f\in\B$  belongs to the class $\mathbf{CL}$, if
there exists a constant $C>0$ such that $|f(z_1)-f(z_2)|\le C |z_1-z_2|$ for any $z_1, z_2\in\R$.
\end{Definition}

Finally, we will also need the following process:
\begin{equation}\label{functI}
\V(r,s,t,z):=\int_r^s\int_{\R} p_{t-t'}(z-z') \,W(dt',dz'), \quad 0\le r\le s \le t,\,\,z\in\R.
\end{equation}
We see that, by definition, $V(s,t,z)=\V(s,t,t,z)$.

\subsection{Estimates involving Gaussian density}

First, let us formulate a number of very simple lemmas, whose proofs will be given later in the Appendix.

\begin{Lemma}\label{L:densest}
For any $T>0$, $\delta\ge0$ there exists $C=C(T,\delta)>0$ such that for any $s,t\in[0,T]$,  we have
\begin{equation*}
\int_\R|p_t(z)-p_s(z)|\,\,(|z|^{\delta}\vee1)\,dz\le C|\log t -\log s|.
\end{equation*}
\end{Lemma}

\begin{Lemma}\label{L:GE:3der}
For any $\delta\in(0,2/3)$, $T>0$ there exists $C=C(T,\delta)>0$ such that for any
$0<t_1<t_2<t$, $z_1,z_2,z\in\R$ we have
\begin{align}
&\int_\R p_{t}(z-z') \Lambda_\delta(|z'|\vee 1)dz'\le C \Lambda_\delta(|z|\vee1)\label{firstestapp};\\
&\int_\R\int_{t_1}^{t_2}  \bigl| \frac{\d }{\d t'}p_{t-t'}(z-z')\bigr| (t_2-t')^{2/3-\delta}\Lambda_\delta(|z'|\vee 1)\, dt'dz'\le
C |t_2-t_1|^{2/3-\delta}\Lambda_\delta(|z|\vee1);\label{seqestapp} \\
&\int_\R |p_{t}(z_1-z')-p_{t}(z_2-z')| \Lambda_\delta(|z'|\vee 1)dz'\le C t^{-1/2}|z_1-z_2| \Lambda_\delta(|z_1|\vee|z_2|\vee1);\label{thirdstapp}\\
&\int_\R\int_{t_1}^{t_2}  \bigl| \frac{\d }{\d t'}(p_{t-t'}(z_1-z')-p_{t-t'}(z_2-z'))\bigr| (t_2-t')^{2/3-\delta}\Lambda_\delta(|z'|\vee 1)\, dt'dz'\nnn\\
&\quad\quad\le C  (t_2-t_1)^{2/3-\delta}(t-t_1)^{-1/2} |z_1-z_2| \Lambda_\delta(|z_1|\vee|z_2|\vee1).\label{fourthstapp}
\end{align}
\end{Lemma}

\begin{Lemma}\label{L:contextraterm} Let $f\colon [0,\infty)\times\R\to \R$ be a bounded measurable function. Define
\begin{equation*}
h(t,z):=\int_0^t\int_{\R} p_{t-t'}(z-z')f(t',z')\,dz'\,dt',\quad z\in\R,\,t\ge 0
\end{equation*}
Then for any $T>0$, $\delta>0$, there exists a constant $C=C(T,\delta)$ such that for
any $t_1,t_2\in[0,T]$, $z_1,z_2\in\R$ we have
\begin{equation}\label{LipH}
|h(t_1,z_1)-h(t_2,z_2)|\le C \|f\|_\infty(|z_1-z_2|+|t_1-t_2|^{1-\delta}).
\end{equation}
\end{Lemma}

\begin{Lemma}\label{L:timeq} Let $\mu>0$ and let $q\in\mathbf{B}(\mu)$, that is, for some $M>0$ we have $|q(z)|\le M (|z|^\mu\vee1)$, $z\in\R$. Then for any $T>0$ there exists a constant $C=C(T,\mu)$ such that function
\begin{equation*}
h(t,z):=\int_{\R} p_{t}(z-z')q(z')\,dz'\quad z\in\R,\,t\in[0,T]
\end{equation*}
belongs to the class $\mathbf{H}_T(1,1,\mu,CM)$. The constant $C$ does not depend on the function $q$.

Further, if $q\in\mathbf{CL}$, then there exists a constant $C_1=C_1(T,q)$ such that for
any $t_1,t_2\in[0,T]$, $z_1,z_2\in\R$ we have
\begin{equation*}
|h(t_1,z_1)-h(t_2,z_2)|\le C_1(|z_1-z_2|+|t_1-t_2|^{1/2}).
\end{equation*}

\end{Lemma}

\subsection{Existence of a regular version of $V$}

The next lemma establishes the global regularity properties of the noise process $\V$ (recall the definition of $\V$ given in \eqref{functI}).

\begin{Lemma}\label{L:GE:leon}
There exists a set $\Omega_V\subset\Omega$ with the following properties:
\begin{enumerate}[label={\arabic*$)$}]
\item
$\P(\Omega_V)=1$;
\item
Let $\omega\in\!\Omega_V$. Then the functions $(s,t,z)\mapsto\!V(s,t,z,\omega)$ and $(r,s,t,z)\mapsto\!\V(r,s,t,z,\omega)$ are continuous. Furthermore, for any $T>0$, $\eps\in(0,1/2)$, $p>0$ there exists $K(\omega)=K(\omega,p,T,\eps)$ such that for any
$0\le s\le t\le T$, $0\le s<t_1<t_2\le T$, $z,z_1,z_2\in\R$ we have
\begin{align}
|\V(0,s,t,z_1)-\V(0,s,t,z_2)|&\le  K(\omega) |z_1-z_2|^{1/2-\eps}(|z_1|^{2\eps}\vee |z_2|^{2\eps}\vee 1)\label{lipshzet};\\
|\V(0,s,t_1,z)-\V(0,s,t_2,z)|&\le  K(\omega) |t_1-t_2| (t_1-s)^{-1}(|z|^{\eps}\vee  1)\label{lipshzet2};\\
|V(s,t,z)|&\le  K(\omega) (|z|^\eps\vee 1)\label{stbound}.
\end{align}
Moreover, $\omega\mapsto K(\omega)$ is a random variable with $\E K(\omega)^p<\infty$.
\end{enumerate}
\end{Lemma}

Before we start the proof, let us emphasize that the result is of course not surprising:
it is well-known that the convolution of the white noise with the heat kernel is \textit{locally} H\"older $(1/2-\eps)$ continuous in space
and H\"older $(1/4-\eps)$ continuous in time (see, e.g., \cite[Exercise~6.9]{Khosh}). The lemma gives uniform \emph{global}
control on H\"older coefficients. As one can expect, the price to pay is that H\"older coefficients are no longer bounded but grow as $|z|^{\eps}$ if $z\to\infty$ (in other words,  the convolution of the white noise with the heat kernel is not a \emph{globally} H\"older function).
\begin{proof}
First, we consider the case $s=t$. In this case $\V(0,t,t,z)=V(0,t,z)$ Arguing as in the proof of \cite[Theorem~6.7]{Khosh}, we derive for any $p\ge2$, $T>0$, $t_1,t_2\in[0,T]$, $z_1,z_2\in\R$,
\begin{equation*}
\E |V(0,t_1,z_1)-V(0,t_2,z_2)|^p\le C(|t_1-t_2|^{1/4}+|z_1-z_2|^{1/2})^p.
\end{equation*}
for some $C=C(p,T)>0$. It follows from the Kolmogorov continuity criterion in the nonhomogeneous form (see, e.g., \cite[Theorem~1.4.1]{Kuni}) that for any $\eps\in(0,1-6/p)$ there exists a constant $C_1=C_1(C,p,T,\eps)>0$ such that for any $N>0$ there exist a set $\Omega_{N,p,T,\eps}\subset\Omega$ with $\P(\Omega_{N,p,T,\eps})=1$ and a random variable $K_{N,p,T,\eps}$  with $\E K_{N,p,T,\eps}^{p}\le C_1$ such that for any $\omega\in\Omega_{N,p,T,\eps}$, $t_1,t_2\in[0,T]$,  $z_1,z_2\in[-N,N]$ we have for a version of $V$ (that with some abuse of notation will be denoted further by the same letter)
\begin{equation*}
|V(0,t_1,z_1)-V(0,t_2,z_2)|\le K_{N,p,T,\eps}(\omega)N^{\frac12(6/p+\eps)}(|t_1-t_2|^{\frac14(1-6/p-\eps)}+|z_1-z_2|^{\frac12(1-6/p-\eps)}).
\end{equation*}
Put $K_{p,T,\eps}:=\sup_{N\in \N}K_{N,p,T,\eps} N^{-2/p}$. Clearly,
\begin{equation*}
\E K_{p,T,\eps}^p\le \sum_{N=1}^\infty \E K_{N,p,T,\eps}^p N^{-2}\le C_1\sum_{N=1}^\infty N^{-2}<\infty.
\end{equation*}
Hence the random variable $K_{p,T,\eps}$ is almost surely finite. Put
\begin{equation*}
\Omega_V:=\bigcap_{N,p,T,\eps} (\Omega_{N,p,T,\eps}\cap\{K_{p,T,\eps}<\infty\}),
\end{equation*}
where the intersection is taken among all $N\in\N$, and rational $p\ge2$, $\eps\in(0,1-6/p)$, $T>0$. Clearly, $\P(\Omega_V)=1$.
We see that for any $T>0$, $p\ge2$, $\eps\in(0,1-6/p)$ there exists a random variable $\wt K=\wt K_{p,T,\eps}$ such that
$\E \wt K_{p,T,\eps}^p<\infty$ and
\begin{multline}\label{Vbound}
|V(0,t_1,z_1)-V(0,t_2,z_2)|\\
\le \wt K(\omega)(|z_1|\vee|z_2| \vee 1)^{5/p+\eps/2} (|t_1-t_2|^{\frac14(1-6/p-\eps)}+|z_1-z_2|^{\frac12(1-6/p-\eps)})
\end{multline}
for all $\omega\in\Omega_V$, $t_1,t_2\in[0,T]$, $z_1,z_2\in\R$. By setting $z_1=z_2=z$, $t_1=t$, $t_2=0$, we get
\begin{equation}\label{tbound}
|V(0,t,z)|\le  \wt K(\omega) |t|^{\frac14(1-6/p-\eps)}(|z|\vee 1)^{5/p+\eps/2}.
\end{equation}

Now we consider the general case. By above,
\begin{align*}
|\V(0,s,t,z_1)-\V(0,s,t,z_2)|&=\Bigl|\int_{\R} p_{t-s}(z')(V(0,s,z_1-z')-V(0,s,z_2-z'))\,dz'\Bigr|\\
&\le C_2 \wt K(\omega) |z_1-z_2|^{\frac12(1-6/p-\eps)}(|z_1|\vee |z_2|\vee 1)^{5/p+\eps/2}.
\end{align*}
This implies \eqref{lipshzet}. Similarly, by Lemma~\ref{L:densest} and \eqref{tbound}
\begin{align*}
|\V(0,s,t_1,z)-\V(0,s,t_2,z)|&=\Bigl|\int_{\R} (p_{t_1-s}(z')-p_{t_2-s}(z')) V(0,s,z-z')\,dz'\Bigr|\\
&\le C_3 \wt K(\omega) |t_1-t_2|(t_1-s)^{-1}(|z|\vee 1)^{5/p+\eps/2}.
\end{align*}
This yields \eqref{lipshzet2}. To show continuity of $V$ we note that the function $(t,z)\mapsto V(0,t,z,\omega)$ is continuous
for $\omega\in\Omega_V$ by inequality \eqref{Vbound}. Further, for $0\le s \le t$ we have
\begin{equation}\label{whatisV}
V(s,t,z)=V(0,t,z)-\V(0,s,t,z)=V(0,t,z)+\int_{\R} p_{t-s}(z') V(0,s,z-z')\,dz'.
\end{equation}
Hence $V$ is continuous as a sum of two continuous functions. Since
\begin{equation*}
\V(r,s,t,z)=\V(r,t,t,z)-\V(s,t,t,z)=V(r,t,z)-V(s,t,z),\quad 0\le r\le s\le t,
\end{equation*}
we see that $\V$ is also continuous. Finally, bound \eqref{tbound} combined with
\eqref{whatisV} implies \eqref{stbound}.
\end{proof}

\subsection{Continuity lemmas}

As we mentioned before, in order to prove Theorem~\ref{T:exandpbp} we approximate the drift in equation \eqref{SPDEmildflow} by a sequence
of piecewise continuous functions and pass to the limit (see the proof of Lemma~\ref{L:212} below). If the function $b$ were continuous,
this would not require any additional clarifications. However in our case when we assume that $b$ is just a measurable bounded function
we need to explain why the passage to the limit is justified here.

Recall the definition of the approximation operator $\lambda_n$  given in \eqref{lambda}.

\begin{Lemma}\label{L:cont-prev}
For any $\eps>0$, $M>0$, $h>1/2$, $N\in\N$, $T>0$, $\mu>0$  there exists $\delta>0$ such that for each open set $U\subset\R$ with $|U|\le\delta$
we have with probability greater or equal than $1-\eps$
\begin{equation}\label{mest}
\int_{0}^{T} \I_U(V(t+s,z,\omega)+f_1(t,z)+\lambda_r(f_2)(t,z))\,dt\le\eps.
\end{equation}
simultaneously for all $z\in[-N,N]$, $r\in\N$, $s\in[0,T]$  and all $f_1,f_2\in \mathbf{H}_T(h,1,\mu,M)$.
\end{Lemma}

The proof of the lemma is given in Section~\ref{S:PCL}.

\begin{Lemma}\label{L:cont_gen}
Let $b\in\B$. Then there exists a set $\Omega_C\subset \Omega$ with the following properties:
\begin{enumerate}[label={\arabic*$)$}]
\item
$\P(\Omega_C)=1$;
\item
Let $\omega\in\Omega_C$. Then for any $T>0$, $h>1/2$, $M>0$, $\mu>0$, $0\le t_1\le t_2\le T$, $z\in\R$,  $\theta\in\B$, function $\psi\in\mathbf{H}_T(h,1,\mu,M)$, any sequence $(s_n)_{n\in\Z_+}$, $s_n\in[0,T]$ converging to $s$ and any sequence of functions $f_n\in\mathbf{H}_T(h,1,\mu,M)$ converging pointwise on $(0,T]\times\R$ to a limit $f$ we have
\begin{align*}
\lim_{n\to\infty}\int_{t_1}^{t_2} \theta(t)b(V(t+s_n,z,&\omega)+f_n(t,z)+\lambda_n(\psi)(t,z))\,dt\nnn\\
&= \int_{t_1}^{t_2} \theta(t)b(V(t+s,z,\omega)+f(t,z)+\psi(t,z))\,dt.
\end{align*}
\end{enumerate}
\end{Lemma}

\begin{proof}
The proof is based on the ideas from the proofs of \cite[Lemmas~3.3,~3.4]{Dav}. Fix $h>1/2$, $\mu>0$ and integers $N,M,T>0$. Take arbitrary $\eps>0$.
Then there exists $\delta>0$ such that statement of Lemma~\ref{L:cont-prev} is satisfied.

By Lusin's theorem there exists a continuous bounded function $\widetilde b\colon\R\to\R$ and an open set $U$ such that $|U|<\delta$, $\|\widetilde b\|_\infty\le 2 \|b\|_\infty$ and
$b(x)=\widetilde b(x)$ for all $x\notin U$. Thus, we have the bound
\begin{equation}\label{bbound}
|b(x)-\widetilde b(x)|=\I(x\in U)|b(x)-\widetilde b(x)|\le  3 \|b\|_\infty\I(x\in U).
\end{equation}

Further, by Lemma~\ref{L:cont-prev}, there exists a set $\Omega_\eps$ with $\P(\Omega_\eps)\ge1-\eps$
such that bound \eqref{mest} holds on $\Omega_\eps$. Take now any $\omega\in\Omega_\eps$, $0\le t_1\le t_2\le T$, $s,s_n\in[0,T]$, $s_n\to s$, $z\in[-N,N]$,   $\theta\in\B$, a function $\psi\in\mathbf{H}(h,1,\mu,M)$ and any sequence of functions $f_n\in\mathbf{H}(h,1,1,M)$ converging pointwise to a limit $f\in\mathbf{H}(h,1,\mu,M)$. Taking into account \eqref{bbound}, we have
\begin{align*}
\limsup_{n\to\infty}&\int_{t_1}^{t_2} \theta(t)b(V(t+s_n,z,\omega)+f_n(t,z)+\lambda_n(\psi)(t,z))\,dt\\
\le& \int_{t_1}^{t_2} \theta(t)\widetilde b(V(t+s,z,\omega)+f(t,z)+\psi(t,z))\,dt\\
&+3\|b\|_\infty\|\theta\|_\infty \limsup_{n\to\infty}\int_{t_1}^{t_2} \I_U(V(t+s_n,z,\omega)+f_n(t,z)+\lambda_n(\psi)(t,z))\,dt\\
\le& \int_{t_1}^{t_2} \theta(t)\widetilde b(V(t+s,z,\omega)+f(t,z)+\psi(t,z))\,dt
+3\|b\|_\infty\|\theta\|_\infty \eps\\
\le& \int_{t_1}^{t_2} \theta(t)b(V(t+s,z,\omega)+f(t,z)+\psi(t,z))\,dt+3\|b\|_\infty\|\theta\|_\infty\eps\\
&+3\|b\|_\infty\|\theta\|_\infty \int_{t_1}^{t_2} \I_U(V(t+s,z,\omega)+f(t,z)+\psi(t,z))\,dt\\
\le& \int_{t_1}^{t_2} \theta(t)b(V(t+s,z,\omega)+f(t,z)+\psi(t,z))\,dt+6\|b\|_\infty\|\theta\|_\infty\eps,
\end{align*}
where the second and the last inequalities follow from Lemma~\ref{L:cont-prev}.

By a similar argument, we have on $\Omega_\eps$
\begin{multline*}
\liminf_{n\to\infty}\int_{t_1}^{t_2} \theta(t)b(V(t+s_n,z,\omega)+f_n(t,z)+\lambda_n(\psi)(t,z))\,dt\\
\ge  \int_{t_1}^{t_2} \theta(t)b(V(t+s,z,\omega)+f(t,z)+\psi(t,z))\,dt-6\|b\|_\infty\|\theta\|_\infty\eps.
\end{multline*}
Since $\eps$ was arbitrary, and since $\P(\Omega_\eps)\ge1-\eps$ we see that there exists a set $\Omega(N,M,T,h,\mu)$ such that $\P(\Omega(N,M,T,h,\mu))=1$ and
\begin{multline*}
\lim_{n\to\infty}\int_{t_1}^{t_2} \theta(t)b(V(t+s_n,z,\omega)+f_n(t,z)+\lambda_n(\psi)(t,z))\,dt\\
=  \int_{t_1}^{t_2} \theta(t)b(V(t+s,z,\omega)+f(t,z)+\psi(t,z))\,dt.
\end{multline*}
for any $\omega\in\Omega(N,M,T,h,\mu)$, $0\le t_1\le t_2\le T$, $s,s_n\in[0,T]$, $s_n\to s$, $\theta\in\B$, function $\psi\in\mathbf{H}(h,1,\mu,M)$,
$z\in[-N,N]$ and any sequence of functions $f_n\in\mathbf{H}(h,1,\mu,M)$ converging pointwise to a limit $f\in\mathbf{H}(h,1,\mu,M)$.

To complete the proof of the lemma it remains to take $\Omega_C:=\cap\Omega(N,M,T,h,\mu)$ where the intersection is over all positive integers $N,M,T$ and rational $h>1/2$, $\mu>0$.
\end{proof}
\section{Proofs of Theorem~\ref{T:exandpbp} and Theorem~\ref{T:stochasticflow}(b)}\label{S:5}

Most of the section is devoted to the proof of Theorem~\ref{T:exandpbp}. Fix a bounded measurable function $b$. Without loss of generality and to ease the notation we assume in this section that $\|b\|_\infty\le1$. Now with such $b$ at hand we take for the rest of the section
\begin{equation}\label{Omegaprime}
\Omega':=\Omega_E\cap\Omega''\cap\Omega_V\cap\Omega_C\subset \Omega,
\end{equation}
where the set $\Omega_E$ is defined in the sketch of the proof of Theorem~\ref{T:exandpbp} in Section~\ref{S:MR}, $\Omega''$ is from Theorem~\ref{T:gladkost}, $\Omega_V$ is from regularity Lemma~\ref{L:GE:leon} and $\Omega_C$ is from continuity Lemma~\ref{L:cont_gen}.
Thus, on $\Omega'$ the statements of the aforementioned
theorems and lemmas are satisfied and $\P(\Omega')=1$.

We begin this section with an easy observation.

\begin{Proposition}\label{P:51}
Let $s\ge0$, $q\in\mathbf B(0+)$. Let $u_{s,q}$ be any solution to \eqref{SPDEmildflow} that starts with
initial condition $q$ at time $s$. Then $u_{s,q}(t,\cdot,\omega)\in \mathbf B(0+)$ for any $\omega\in\Omega'$, $t\ge s$.
\end{Proposition}
\begin{proof}
This statement immediately follows from equation \eqref{SPDEmildflow} and estimate \eqref{stbound}.
\end{proof}

\subsection{Existence part of Theorem~\ref{T:exandpbp}}\label{S:exist}

In this subsection we present the proof of the existence part of Theorem~\ref{T:exandpbp}. Our main tool is the following lemma
that establishes continuity of solution to \eqref{SPDEmildflow} with respect to the initial condition. Recall that the set $\Omega'$ is defined in \eqref{Omegaprime}.

\begin{Lemma}\label{L:contincond}
Let $\omega\in\Omega'$. Let $(s_n)_{n\in\Z_+}$, $s_n\ge0$ be a sequence that converges to $s$, as $n\to\infty$. Let $(q_n)_{n\in\Z_+}$ be a sequence of measurable functions $\R\to\R$ such that $q_n(z)\to q(z)$ as $n\to\infty$ Lebesgue--almost everywhere in $z$. Assume that there exist $C>0$, $\mu>0$ such that for any  $n\in\Z_+$ one has
\begin{equation*}
|q_n(z)|\le C (|z|^\mu\vee 1),\quad z\in\R.
\end{equation*}
For each $n\in\Z_+$ let  $u_{s_n,q_n}(\cdot,\cdot,\omega)$ be a solution to \eqref{SPDEmildflow} that starts with the initial condition $q_n$ at time $s_n$.

Then there exists a solution $u_{s,q}(\cdot,\cdot,\omega)$ to \eqref{SPDEmildflow} that starts with the initial condition $q$ at time $s$. Moreover, there exists a subsequence $(n_k)_{k\in\Z_+}$ such that for any  $t>0$, $z\in\R$ we have
\begin{equation*}
u_{s_{n_k},q_{n_k}}(s_{n_k}+t,z,\omega)\to u_{s,q}(s+t,z,\omega)\quad\text{as }\,k\to\infty.
\end{equation*}
\end{Lemma}

This lemma implies that for any $\omega\in\Omega'$ a sequence of solutions to equation \eqref{SPDEmildflow} that start at time $s_n$ from the initial condition $q_n$ has a subsequence that
converges pointwise to a solution of equation \eqref{SPDEmildflow} that starts at time $s$ from the initial condition~$q$.

\begin{proof}
Fix $\omega\in\Omega'$, the sequences $(s_n)_{n\in\Z_+}$, $(q_n)_{n\in\Z_+}$ as in the lemma and also any $T>0$. By the definition of $u^*_{s_n,q_n}$ (recall equation \eqref{ushifted}), we have for any $t\in (0,T]$, $z\in\R$
\begin{equation*}
u^*_{s_n,q_n}(t,z)=\int_{\R}p_{t}(z-z')q_n(z')\,dz'+\int_0^t\int_{\R} p_{t-t'}(z-z')b(u^*_{s_n,q_n}(t',z'))\,dz'\,dt'+V(s_n,t+s_n,z)
\end{equation*}
and $u^*_{s_n,q_n}(0,z)=q_n(z)$. For $n\in\Z_+$ set
\begin{equation*}
h_n(t,z):=\int_0^t\int_{\R} p_{t-t'}(z-z')b(u^*_{s_n,q_n}(t',z'))\,dz'\,dt',\quad t\in [0,T],\,z\in\R.
\end{equation*}
Clearly, the sequence $(h_n)_{n\in\Z_+}$ is uniformly bounded. Indeed, $\|h_n\|_\infty\le T \|b\|_\infty$. It follows from Lemma~\ref{L:contextraterm} that for any $t_1,t_2\in[0,T]$, $z_1,z_2\in\R$
\begin{equation}\label{LipHh}
|h_n(t_1,z_1)-h_n(t_2,z_2)|\le  C \|b\|_{\infty} (|z_2-z_1|+|t_2-t_1|^{3/4})
\end{equation}
for some $C=C(T)>0$ that is independent of $n$. Hence the  Arzel\`{a}--Ascoli theorem for locally compact metric spaces (see, e.g., \cite[Theorem~4.44]{Fol}) implies that there exists a subsequence $(n_k)_{k\in\Z_+}$,
such that $h_{n_k}$ converges pointwise to some function $h$. We simplify the notation by assuming that we have already started with such a subsequence and that $n_k=k$. Hence, \eqref{LipHh} yields
\begin{equation*}
|h(t_1,z_1)-h(t_2,z_2)|\le  C \|b\|_{\infty} (|z_2-z_1|+|t_2-t_1|^{3/4}),\quad t_1,t_2\in[0,T],\,\, z_1,z_2\in\R.
\end{equation*}
Put
\begin{align}\label{ustardef}
u^*_{s,q}(t,z)&:=\int_{\R}p_{t}(z-z')q(z')\,dz'+h(t,z)+V(s,t+s,z),\quad  t\in (0,T],\,z\in\R,\\
u^*_{s,q}(0,z)&:=q(z),\quad z\in\R\nnn.
\end{align}
We claim now
that $u^*_{s,q}$ is a solution to  \eqref{shiftSPDEmildflow} on $[0,T]$. Indeed,
we observe that $u^*_{s_n,q_n}(t,z)$ can be written as follows:
\begin{equation*}
u^*_{s_n,q_n}(t,z)=V(t+s_n,z)+g_n(t,z),
\end{equation*}
where
\begin{equation}\label{gdef}
g_n(t,z):=\int_{\R}p_{t}(z-z')q_n(z')\,dz'+h_n(t,z)-\V(0,s_n,t+s_n,z),\quad  t\in [0,T],\,z\in\R.
\end{equation}
It follows from Lemma~\ref{L:timeq}, Lemma~\ref{L:GE:leon} and inequality \eqref{LipHh}
that there exists $M>0$ such that for any $n\in\Z_+$ the function $g_n\in\mathbf{H}_T(3/4,1,\mu,M)$.
By our assumptions and the dominated converging theorem, the first term at the right-hand side of \eqref{gdef}
converges pointwise to $\int_{\R}p_{t}(z-z')q(z')\,dz'$ for $(t,z)\in(0,T]\times\R$. By Lemma~\ref{L:GE:leon},
$\V(0,s_n,t+s_n,z)\to \V(0,s,t+s,z)$ as
$n\to\infty$ for  $(t,z)\in[0,T]\times\R$. This together with
$h_n$ converging pointwise to $h$ implies that
%
%
\begin{equation*}
\lim_{n\to\infty}g_n(t,z)=\int_{\R}p_{t}(z-z')q(z')\,dz'+h(t,z)-\V(0,s,t+s,z)=:g(t,z),\quad t\in (0,T],\,z\in\R.
\end{equation*}
Therefore,
\begin{align*}
h(t,z)&=\lim_{n\to\infty}h_n(t,z)\\
&=\lim_{n\to\infty} \int_{\R}\int_0^t p_{t-t'}(z-z')b(u^*_{s_n,q_n}(t',z'))\,dt'\,dz'\\
&=\lim_{n\to\infty} \int_{\R}\int_0^t p_{t-t'}(z-z')b(V(t+s_n,z)+g_n(t,z))\,dt'\,dz'
\end{align*}
Note that the function $b$ is not necessarily continuous and we cannot pass to the limit directly. Therefore to pass to the limit we employ Lemma~\ref{L:cont_gen} with the following set of parameters: $f_n\leftarrow g_n$, $f\leftarrow g$, $\psi\leftarrow0$, $\theta\leftarrow p_{t-\cdot} (z-z')$, $t_2\leftarrow t$, $t_1\leftarrow 0$. Since $g_n\in\mathbf{H}_T(3/4,1,\mu,M)$ and since for fixed $t$, $z\neq z'$ the function $p_{t-\cdot} (z-z')$ is bounded, we see that all conditions of Lemma~\ref{L:cont_gen} are satisfied. We apply the dominated convergence theorem (this is possible due to the fact that $b$ is bounded) and continue the identity above as follows:
\begin{align}\label{formulah}
h(t,z)&=\int_{\R}\int_0^t p_{t-t'}(z-z')b(V(t+s,z)+g(t,z))\,dt'\,dz'\nnn\\
&=\int_{\R}\int_0^t p_{t-t'}(z-z')b(u^*_{s,q}(t',z'))\,dt'\,dz',
\end{align}
where we also used that by \eqref{ustardef} $u^*_{s,q}(t,z)=V(t+s,z)+g(t,z)$. Obtained identity \eqref{formulah}, combined with \eqref{ustardef}, implies that $u^*_{s,q}$ is indeed a solution to \eqref{shiftSPDEmildflow}. Hence the function
$u_{s,q}(t,\cdot):= u^*_{s,q}(t+s,\cdot)$ solves equation \eqref{SPDEmildflow} that starts with the initial condition $q$ at time $s$.

We note that the convergence of $h_n$ to $h$ and continuity of $V$ imply that
\begin{equation}\label{limfixedt}
\lim_{n\to\infty }u_{s_n,q_n}(s_n+t,z)=\lim_{n\to\infty }u^*_{s_n,q_n}(t,z)=
u^*_{s,q}(t,z)=u_{s,q}(s+t,z).
\end{equation}
for any $t\in(0,T]$, $z\in\R$. Finally, by the standard diagonalization argument, we see that  there exists a subsequence
$(n_k)$ such that identity \eqref{limfixedt} is valid for any $t\in(0,\infty)$.
\end{proof}

\smallskip

\begin{proof}[Proof of existence part of Theorem~\ref{T:exandpbp}]

We recall that we have already fixed $\Xi$, a countable dense subset of $\C_0(\R)$, and $\Theta$, a countable dense subset of $\R_+$ (see the sketch of the proof of Theorem~\ref{T:exandpbp} in Section~\ref{S:MR}). Since $\Omega'\subset\Omega_E$, we see that for any $\omega\in\Omega'$, $s\in\Theta$, $q\in\Xi$ equation \eqref{SPDEmildflow} has a solution that starts with the initial condition $q$ at time $s$.

Fix any $\omega\in\Omega'$.
Let $q$ now be  an arbitrary element of $\mathbf{B}(0+)$, let $s\in\R$. Let $(q_n)_{n\in\Z_+}$ be a sequence of elements
in  $\Xi$ that converge Lebesgue--almost everywhere
to $q$ and such that for some $C>0$, $\mu>0$ one has $q_n(z)\le C (|z|\vee 1)^\mu$ uniformly over all $n$.  The existence of such a sequence is clear and can be shown by the standard argument. Let $(s_n)_{n\in\Z_+}$ be a sequence of elements in $\Theta$ that converges to $s$. By above, equation \eqref{SPDEmildflow} has a solution that starts with the initial condition $q_n$ at time $s_n$. Hence, by Lemma~\ref{L:contincond}, equation \eqref{SPDEmildflow} has a solution that starts with the initial condition $q$ at time $s$.

Since $q$ and $s$ were arbitrary elements of $\mathbf{B}(0+)$ and $\R_+$, respectively, this concludes the proof of the existence part of Theorem~\ref{T:exandpbp}.
\end{proof}

\subsection{Uniqueness part of Theorem~\ref{T:exandpbp}}\label{S:uniq}

Recall that by Remark~\ref{R:CaptObvious} it is sufficient to show that on $\Omega'$ equation \eqref{shiftSPDEmildflow} has a unique solution. This will straightforwardly imply that the original equation \eqref{SPDEmildflow}
has also a unique solution on $\Omega'$.

Till the end of this section we fix arbitrary $\omega\in\Omega'$, $s\ge0$, $q\in\mathbf B(0+)$. Without loss of generality, we assume $s\in[0,1]$. Let $v$ and $w$ be any two solutions to \eqref{shiftSPDEmildflow} with the initial condition $q$ for our fixed $\omega$, $s$.  To prove the theorem it is sufficient to show that $v(t,z)=w(t,z)$ for $z\in\R$, $t\in[0,T]$ for any $T>0$. We will verify this statement for $T=1$; the proof for other values of $T$ is exactly the same.

We observe that
\begin{equation*}
v(t,z)-w(t,z)=\int_0^t\int_{\R} p_{t-t'}(z-z')(b(v(t',z'))-b(w(t',z')))\,dz'\,dt',\quad t\in[0,1],\,z\in\R.
\end{equation*}
We denote $\psi(t,z):=v(t,z)-w(t,z)$ and rewrite the above equation in the following form:
\begin{equation*}
\psi(t,z)=\int_0^t\int_{\R} p_{t-t'}(z-z')(b(w(t',z')+\psi(t',z))-b(w(t',z')))\,dz'\,dt',\quad t\in[0,1],\,z\in\R.
\end{equation*}
It is easy to check, that for any $r\in[0,1]$ the function $\psi$ also satisfies a more general equation
\begin{align}\label{newpsi2}
\psi(t,z)&=\int_{\R} p_{t-r}(z-z')\psi(r,z')\,dz'\nonumber\\
&+\int_{\R}\int_r^t p_{t-t'}(z-z')\bigl(b(w(t',z')+\psi(t',z'))-b(w(t',z'))\bigr)\, dt'dz',\quad t\in[r,1],\,z\in\R,\\
\psi(0,z)&=\varphi(z)\nonumber.
\end{align}

Our goal is to show that the only solution to this equation with the initial condition $\varphi(z)=0$ is identically zero. This would imply uniqueness of solution to \eqref{shiftSPDEmildflow} for any $\omega\in\Omega'$, $q\in\mathbf{B}(0+)$, $s\in[0,1]$.  To show this we have to analyze this equation with a more general class of initial conditions. Namely, we assume that the function $\varphi\in\mathbf{CL}$ (recall that the class $\mathbf{CL}$ is introduced in Definition~\ref{D:BL}). Note also that the functions $\psi$, $v$, $w$ depend also on fixed $\omega$, $s$, $q$. In order not to over-crowd the notation, we write $\psi(t,z)$ for $\psi(t,z,\omega,s,q)$ and so on.

To show that equation \eqref{newpsi2} has only a trivial solution we will need to control the norm of $\psi(t,\cdot)$. We will work with a weighted H\"older norm. The use of a weighted norm is natural here since we work with functions defined on a noncompact space $\R$. Thus, for a function  $f\colon\R\to\R$ we put
\begin{equation*}
\|f\|_w:=\sup_{z\in\R} |f(z)|e^{-|z|}.
\end{equation*}
For $\delta>0$ consider a weighted Lipschitz coefficient of $f$
\begin{equation*}
Lip_\delta(f):=\sup_{z_1\neq z_2} \frac{|f(z_1)-f(z_2)|}{|z_1-z_2|\Lambda_\delta(|z_1|\vee |z_2|\vee1)}.
\end{equation*}
Recall that the function $\Lambda_\delta$ was defined in \eqref{LambdaL}.

Finally, define a weighted H\"older norm of $f$ by
\begin{equation*}
\|f\|_{1,\delta}:=\|f\|_w+Lip_\delta(f).
\end{equation*}
We have to use the additional factor $z^\delta$ in the weight of the Lipschitz coefficient because this
factor appears in the right-hand side of our main bound \eqref{predmainest} in Theorem~\ref{T:gladkost} (see also
Lemma~\ref{C:31} below).

The function $w$ can be represented as $w(t,z)=V(t+s,z)+g(t,z)$, where
\begin{equation}\label{ggg}
g(t,z):=\int_{\R}p_t(z-z')q(z')\,dz'+\int_0^t\int_{\R} p_{t-t'}(z-z')b(w(t',z'))\,dz'\,dt'-\V(0,s,t+s,z),
\end{equation}
where $t\in[0,1]$, $z\in\R$ and $\V$ was defined in \eqref{functI}. We used here the identity
\begin{equation*}
V(s,t+s,z)=\V(s,t+s,t+s,z)=\V(0,t+s,t+s,z)-\V(0,s,t+s,z).
\end{equation*}

The next two lemmas establish useful properties of the functions $g$ and $\psi$.
\begin{Lemma}\label{L:baza}
The function $g$ defined in \eqref{ggg} belongs to the class $\mathbf H(1-,1,0+)$.
\end{Lemma}
\begin{proof}
The statement is an immediate corollary of Lemma~\ref{L:contextraterm}, Lemma~\ref{L:timeq} and Lemma~\ref{L:GE:leon}.
\end{proof}
\begin{Remark}\label{R:strannayaremarka}
Note that at any time $t>s$ a solution to \eqref{SPDEmildflow}, $u_{s,q}(t,\cdot)$, is much more regular than its initial condition $q\in\mathbf{B}(0+)$. Namely, $u_{s,q}(t,\cdot)$ is a H\"{o}lder function with exponent $1/2-$. If one starts with such a ``regular'' initial condition $q$ (H\"{o}lder  with exponent $1/2-$),  then it is possible to show that $g$ is more regular than it is shown in Lemma~\ref{L:baza}. Namely, $g\in \mathbf H(1-,3/4,0+)$. However we will not use this improvement of regularity of $g$ in our proof and will continue
to consider solutions to \eqref{SPDEmildflow} that start from the initial condition $q\in\mathbf{B}(0+)$.
\end{Remark}
\begin{Lemma}\label{L:30}
Assume that $\varphi\in\mathbf{CL}$. Then any solution $\psi$ to \eqref{newpsi2} is a bounded function on $[0,1]\times\R$, which is Lipschitz in space and H\"older in time with exponent $1/2$. That is, there exists a constant $C=C(\varphi)$ such that
\begin{align}
\sup_{\substack{t\in[0,1]\\z\in\R}}&|\psi(t,z)|\le C\label{boundts}\\
\sup_{\substack{t_1,t_2\in[0,1]\\z_1,z_2\in\R}}&\frac{|\psi(t_1,z_1)-\psi(t_2,z_2)|}{ |t_1-t_2|^{1/2}+|z_1-z_2|}\le C.\label{Liptimespace}
\end{align}
In particular, $\psi(t,\cdot)\in\mathbf{CL}$ for any $t\in[0,1]$.
\end{Lemma}
\begin{proof}
Take in \eqref{newpsi2} $r=0$. Then bound \eqref{boundts} obviously follows from boundedness of the functions $b$ and $\varphi$.
Estimate \eqref{Liptimespace} is obtained by a straightforward application of  Lemma~\ref{L:contextraterm} and Lemma~\ref{L:timeq}.
\end{proof}

The next lemma gives ``smoothing'' properties of the operator
\begin{equation}\label{operator}
(x(\cdot),s,t,z)\mapsto \int_{\R}\int_{0}^{s} p_{t-t'}(z-z')b(V(t'+s,z')+f(t',z')+x(z'))\, dt'dz',\quad x\in\B,\, 0\le s \le t,
\end{equation}
simultaneously for all $f\in\mathbf H(1-,1,0+)$. Recall the definition of the difference of two Gaussian kernels $\Delta p_{t}$ from \eqref{DeltaP}.

\begin{Lemma}\label{C:31}
For any  $\delta\in(0,2/3)$, $N>0$, and any  function $f\in\mathbf H(1-,1,0+)$
there exists a constant $C=C(\omega,N, f,\delta)<\infty$ such
that for any  $0\le t_1\le t_2\le t\le 1$, $s\in[0,1]$, $z,z_1,z_2\in\R$, and any $x,y\in\B$ with  $\|x\|_{\infty},\|y\|_{\infty}\le N$ we have
\begin{align}\label{mestfunc1}
\int_{\R}\Bigl|&\int_{t_1}^{t_2} p_{t-t'}(z-z')\times\nnn\\
&\times\bigl(b( V(t'+s,z')+f(t',z')+x(z'))-b(V(t'+s,z')+f(t',z')+y(z'))\bigr)\, dt'\Bigr|dz'\nonumber\\
\le& C \|x-y\|_{w}|t_2-t_1|^{2/3-\delta}\Lambda_\delta(|z|\vee1);\\
\int_{\R}&\Bigl|\int_{t_1}^{t_2} \Delta p_{t-t'}(z_1-z', z_2-z')\times\nnn\\
&\times\bigl(b(V(t'+s,z')+f(t',z')+x(z'))-b(V(t'+s,z')+f(t',z')+y(z'))\bigr) dt'\Bigr|dz'\nonumber\\
\le& C \|x-y\|_w (t-t_1)^{-1/2}|t_2-t_1|^{2/3-\delta}|z_1-z_2|\Lambda_\delta (|z_1|\vee |z_2| \vee 1).\label{mestfunc2}
\end{align}
\end{Lemma}
\begin{Remark}
If $f\in \mathbf H(1-,3/4,0+)$ (see Remark~\ref{R:strannayaremarka} where it is explained why this is relevant), then the operator \eqref{operator} is more regular in time.
Namely the term $|t_2-t_1|$ in the right-hand side of \eqref{mestfunc1} and \eqref{mestfunc2} has the exponent $3/4-\delta$ instead of $2/3-\delta$.
\end{Remark}
\begin{proof}[Proof of Lemma~\ref{C:31}]
Fix $\delta>0$, $N>0$ and a function $f\in\mathbf{H}(1-,1,0+)$. Consider a function
\begin{equation*}
B(r_1,r_2,s,\alpha, \beta,z):=\int_{r_1}^{r_2} \bigl(b(V(t'+s,z)+f(t',z)+\alpha)-b(V(t'+s,z)+f(t',z)+\beta)\bigr)\, dt',
\end{equation*}
defined for $\alpha,\beta,z \in\R$, $0\le r_1\le r_2\le 1$, $s\in[0,1]$.

It follows from Theorem~\ref{T:gladkost} that there exists a constant $C(\omega,N,f,\delta)$ such that for any $\alpha,\beta,z \in\R$,  $|\alpha|,|\beta|\le N$, $0\le r_1\le r_2\le 1$, $s\in[0,1]$ we have
\begin{equation}\label{bestt}
|B(r_1,r_2,s,\alpha,\beta,z)|\le C(\omega,N,f,\delta)|r_2-r_1|^{2/3-\delta}|\alpha-\beta|(|z|^{\delta}\vee 1).
\end{equation}
To simplify the notation, for the rest of the proof we drop the variables $\omega$, $N$, $f$, $\delta$ and write $C$ instead of
$C(\omega,N,f,\delta)$.

Fix $0\le t_1\le t_2\le 1$. Let $(t',z')\mapsto h(t',z')$, $t'\in [t_1,t_2]$ $z'\in\R$ be a continuously differentiable function in $t'$ for $z'\in\R\setminus E$, where the Lebesgue measure of $E$ is $0$. Then for any $\alpha,\beta\in\R$, $z'\in\R\setminus E$ integration by parts gives
\begin{align*}
\int_{t_1}^{t_2} h(t',z')&\bigl(b(V(t',z')+f(t',z')+\alpha)-b(V(t',z')+f(t',z')+\beta)\bigr)\, dt'\nnn\\
&=-\int_{t_1}^{t_2} h(t',z')\, d_{t'} B(t',t_2,s,\alpha,\beta,z')\nonumber\\
&=h(t_1,z') B(t_1,t_2,s,\alpha,\beta,z')+\int_{t_1}^{t_2}  B(t',t_2,s,\alpha,\beta,z') \frac{\d h }{\d t'}(t',z')\, dt'.
\end{align*}
We integrate over $z'$ and apply estimate \eqref{bestt} to derive for any $x,y\in\B$
\begin{align}\label{genf}
\int_\R\int_{t_1}^{t_2} h(t',z')&\bigl(b(V(t',z')+f(t',z')+x(z'))-b(V(t',z')+f(t',z')+y(z'))\bigr)\, dt'\nnn\\
\le& C \|x-y\|_{w} |t_2-t_1|^{2/3-\delta} \int_\R |h(t_1,z')| \Lambda_\delta(|z'|\vee 1)dz' \nnn\\
&+C \|x-y\|_{w} \int_\R\int_{t_1}^{t_2}  \bigl| \frac{\d h }{\d t'}(t',z')\bigr| |t_2-t'|^{2/3-\delta}\Lambda_\delta(|z'|\vee 1)\, dt'dz'.
\end{align}
For any $t\ge t_2$, $z\in\R$ we can apply this formula to the function $h(t',z'):=p_{t-t'}(z-z')$ (indeed, for $z'\in\R\setminus z$ this function is  continuously differentiable  in $t'$). Using estimates \eqref{firstestapp} and \eqref{seqestapp} from Lemma~\ref{L:GE:3der}, we get \eqref{mestfunc1}. In a similar way, for $t\ge t_2$, $z_1,z_2\in\R$ we apply formula \eqref{genf} to the function  $h(t',z'):=p_{t-t'}(z_1-z')-p_{t-t'}(z_2-z')$. Using estimates \eqref{thirdstapp} and \eqref{fourthstapp} from Lemma~\ref{L:GE:3der}, we obtain \eqref{mestfunc2}.
\end{proof}

\begin{Remark}\label{R:55}
Let us explain how Lemmas \ref{L:baza}, \ref{L:30}, and \ref{C:31} will be used in the proof. Fix initial condition $\varphi\in\mathbf{CL}$. It follows from Lemma~\ref{L:30} that there exists a constant
$C_1=C_1(\varphi)$ such that inequalities \eqref{boundts} and \eqref{Liptimespace} hold. Recall again that the solution $w$ can be represented as $w(t,z)=V(t+s,z)+g(t,z)$, where $g$ was defined in \eqref{ggg}. By Lemma~\ref{L:baza}, $g\in\mathbf H(1-,1,0+)$. Thus, we can apply Lemma~\ref{C:31} with $f\leftarrow g$ and
$N\leftarrow C_{1}$. We see that there exists a constant $C_2=C_2(C_1,\varphi)$ such that the estimates \eqref{mestfunc1} and \eqref{mestfunc2} are satisfied with $C_2$ instead of $C$. We will use further the constant $C_\varphi:=\max(1,C_1,C_2)$ and we will write
$$
V_{g,s}(t,z):=V(t+s,z)+g(t,z).
$$
\end{Remark}

Now, we apply Lemma~\ref{C:31} to analyze the behavior of $\psi$ on a small interval  $[k2^{-m},(k+1)2^{-m}]$.
More precisely, for any $t\in[k2^{-m},(k+1)2^{-m}]$ we will derive bounds on $\|\psi(t,\cdot)\|_{1,\delta}$ in terms of
$\|\psi(k2^{-m},\cdot)\|_{1,\delta}$.  This lemma will be crucial for the whole argument. Namely, we will just apply the bound from Lemma~\ref{L:212} consecutively $2^m$ times to prove later the uniqueness part of Theorem~\ref{T:exandpbp}.

\begin{Lemma}\label{L:212}
For any $\delta\in(0,1/6)$ and any initial condition $\varphi\in\mathbf{CL}$  there exist constants $C=C(\delta,\varphi)$, $m_0=m_0(\delta,\varphi)$ such that for any integers $m>m_0$, $r\in[0,2^m-1]$ we have the following estimate
\begin{equation}
\label{diffonestepres}
\sup_{t\in[\frac{r}{ 2^m},\frac{r+1}{ 2^m}]}\bigl\|\psi(t,\cdot)\|_{1,\delta}\le C \|\psi(\frac{r}{ 2^m},\cdot)\|_{1,\delta}
+Ce^{-2^{m/(2\delta)}}.
\end{equation}
In particular,
\begin{equation}
\label{diffonestepresuseful}
\|\psi(\frac{r+1}{ 2^m},\cdot)\|_{1,\delta}\le C \|\psi(\frac{r}{ 2^m},\cdot)\|_{1,\delta}+Ce^{-2^{m/(2\delta)}}.
\end{equation}

\end{Lemma}

\begin{proof} Fix $\delta\in(0,1/6)$, the initial condition $\varphi$ and integer $m>0$. All the constants that will appear in the proof will depend only on $\delta$ and $\varphi$ but not on~$m$ or $r$. Since $\delta$ is fixed, we will frequently omit the subindex $\delta$ and write $\Lambda(z)$ and $Lip(f)$ instead of $\Lambda_\delta(z)$ and $Lip_\delta(f)$, correspondingly.

To simplify the notation we will show \eqref{diffonestepres} for $r=0$. The proof for other values of $r=1,2,\hdots,2^{m}-1$ is exactly the same.

Recall that we already know from Lemma~\ref{L:30} and Remark~\ref{R:55} that
\begin{equation}\label{ctaunorm}
\sup_{\substack{t_1,t_2\in[0,2^{-m}]\\z\in\R}}
 \frac{|\psi(t_1,z)-\psi(t_2,z)|}{|t_1-t_2|^{1/2}}\le C_{\varphi}.
\end{equation}
This bound is rather rough; our goal is to obtain a much finer bound that we can later iterate over $r$. We will show in he proof that if
$\|\psi(0,\cdot)\|_{1,\delta}$ is small, then the left-hand side of \eqref{ctaunorm} is also very small. This would imply \eqref{diffonestepres}.

Our proof strategy consists of three steps. First, following a standard technique (see, e.g., \cite[proof of Lemma~3.1]{Dav}),
we show that it is sufficient to estimate the supremum in the left-hand side of \eqref{ctaunorm} only for $t_1,t_2\in[0,2^{-m}]$ that are dyadic neighbors. This would imply a corresponding bound for any $t_1,t_2\in[0,2^{-m}]$. As is common in the PDE literature, to get a ``time'' bound we need to obtain first a ``space'' bound. This is done in the second step using approximation technique and estimate~\eqref{mestfunc2}. Finally, using again approximation technique and estimate~\eqref{mestfunc1} we get the required ``time'' bound \eqref{ctaunorm} with much smaller constant.

In the proof of the theorem we will be working with binomial partitions of the interval $[0,1]$.
So, for integers $n\ge0$, $k\in[0,2^n]$ put
\begin{equation}\label{binnum}
t^k_n:=k2^{-n};\quad Lip^k_n:=Lip_\delta(\psi(t^k_n,\cdot)).
\end{equation}
By Lemma~\ref{L:30}, $Lip^k_n$ are finite for any $n\ge0$, $k\in[0,2^n]$.

For the reasons explained above, we study the differences $|\psi(t_1,z)-\psi(t_2,z)|$, $t_1,t_2\in[0,2^{-m}]$ where $t_1$ and $t_2$ are dyadic neighbors. Thus, we define  $\alpha$ as the smallest number such that for any integers $n\ge m$,  $k\in[0,2^{n-m}-1]$ we have
\begin{equation}\label{mainassump22}
\|\psi(\frac{k+1}{2^{n}},\cdot)-\psi(\frac{k}{2^{n}},\cdot)\|_w\le \alpha 2^{-n/2}.
\end{equation}
Note that such an $\alpha$ exists and is finite. As mentioned before, due to Remark~\ref{R:55},  the left-hand side of \eqref{mainassump22} is bounded by $C_{\varphi} 2^{-n/2}$.

Consider a binary notation of $k2^{-n}$. We have $k2^{-n}=\sum_{i=m+1}^{n} d_i 2^{-i}$,
where each $d_i$ equals $0$ or $1$. Define approximations of $k2^{-n}$ by
\begin{equation*}
k_m:=0,\, k_j:=\sum_{i=m+1}^{j} d_i 2^{-i},\quad j\in[m+1,n].
\end{equation*}
It follows from the definition that either $k_{j}=k_{j-1}$ or $k_{j}=k_{j-1}+2^{-j}$. Therefore, we can apply estimate \eqref{mainassump22} $n-m$ times to derive
\begin{align*}\label{cormainassump22}
\|\psi(\frac{k}{2^{n}},\cdot)-\psi(0,\cdot)\|_w\le \sum_{j=m+1}^{n} \|\psi(k_{j},\cdot)-\psi(k_{j-1},\cdot)\|_w\le \alpha\sum_{j=m+1}^{n}2^{-j/2}.
\end{align*}
Hence, there exists $C>0$ such that for any $n\ge m$ and $0\le k\le 2^{n-m}$ we get
\begin{equation*}
\|\psi(\frac{k}{2^{n}},\cdot)\|_w\le \|\psi(0,\cdot)\|_w+C\alpha 2^{-m/2}.
\end{equation*}
Since the function $\psi$ is continuous we get the following bound for any $t\in[0,2^{-m}]$.
\begin{equation}\label{smartbound}
\|\psi(t,\cdot)\|_w\le \|\psi(0,\cdot)\|_w+C\alpha 2^{-m/2}
\end{equation}
Thus, we can effectively bound $\|\psi(t,\cdot)\|_w$ for any $t\in[0,2^{-m}]$. Note that the constant $C$ does not depend on $\alpha$.

To approximate the solutions to \eqref{newpsi2} we consider piecewise approximations defined above in \eqref{lambda}. Namely, we introduce
a sequence of piecewise--constant (in time) functions
\begin{equation*}
\psi_n(\cdot,z):=\lambda_n(\psi(\cdot,z)),\quad z\in\R,
\end{equation*}
where $n\ge m$. We see that $\psi_n(t,z)$ is equal to $\psi((k+1)2^{-n},z)$ for $t\in(k2^{-n},(k+1)2^{-n}]$. In particular, the function
$\psi_m(t,z)$ is constant in $t$ on the interval $(0,2^{-m}]$.

We start with an estimation of the weighted Lipschitz coefficient (with respect to the space variable) of the function $\psi$. We want to do it in all binary points of our initial interval $[0,2^{-m}]$, i.e. in all points of the form $t_n^k=k2^{-n}$. Here $n\ge m$, $0\le k\le 2^{n-m}-1$. We derive from \eqref{newpsi2}
\begin{align}\label{Lipshfirstappr2}
&\psi(t_n^{k+1},z_1)-\psi(t_n^{k+1},z_2)=\int_{\R} p_{2^{-n}}(z')(\psi(t_n^k,z_1-z')-\psi(t_n^k,z_2-z'))\,dz'\nonumber\\
&+\int_{\R}\int_{t_n^k}^{t_n^{k+1}} \Delta p_{t_n^{k+1}-t}(z_1-z',z_2-z')\bigl(b(V_{g,s}(t,z')+\psi(t,z'))-b(V_{g,s}(t,z'))\bigr)\, dtdz'\nonumber\\
&=:I_1+I_2.
\end{align}

First, let us  bound $I_1$. By definition of $Lip_n^k$ (see \eqref{binnum}), we have
\begin{align}\label{Lipshfirstterm}
|I_1|&\le
Lip_n^k |z_1-z_2|\int_{\R} p_{2^{-n}}(z')\Lambda(|z_1-z'|\vee |z_2-z'|\vee1)\,dz'\nnn\\
&\le Lip_n^k |z_1-z_2|\Lambda(|z_1|\vee |z_2|\vee1)\int_{\R} p_{2^{-n}}(z')e^{|z'|}(|z'|^\delta+1)\,dz'
\nonumber\\
&\le  Lip_n^k |z_1-z_2|\Lambda(|z_1|\vee |z_2|\vee1)e^{2^{-n}}(1+C 2^{-n\delta/2}),
\end{align}
where in the second inequality we also used the fact that the function $\Lambda$ is increasing  and
$|z_1-z'|\vee |z_2-z'|\vee1\le (|z_1|\vee |z_2|\vee1)+|z'|$.

To handle $I_2$ we apply continuity Lemma~\ref{L:cont_gen} with the following set of parameters: $f_n\leftarrow g$, $f\leftarrow g$, $\psi\leftarrow \psi$, $\theta\leftarrow \Delta p_{t_n^{k+1}-\cdot}(z_1-z',z_2-z') $, $s_n\leftarrow s$, $t_2\leftarrow t_n^{k+1}$, $t_1\leftarrow t_n^{k}$. Since $g\in\mathbf H(1-,1,0+)$, and since for fixed $z'\neq z_1,z_2$ the function $\Delta p_{t_n^{k+1}-\cdot}(z_1-z',z_2-z')$ is bounded we see that all conditions of the lemma are satisfied. Since the function $b$ is also bounded, we can apply dominated convergence theorem to obtain
\begin{align}\label{I2split}
I_2=&\lim_{l\to\infty} \int_{\R}\int_{t_n^k}^{t_n^{k+1}} \Delta p_{t_n^{k+1}-t}(z_1-z',z_2-z')\bigl(b(V_{g,s}(t,z')+\psi_l(t,z'))-b(V_{g,s}(t,z'))\bigr)\, dtdz'\nonumber\\
=&\int_{\R}\int_{t_n^k}^{t_n^{k+1}} \Delta p_{t_n^{k+1}-t}(z_1-z',z_2-z')\bigl(b(V_{g,s}(t,z')+\psi_n(t,z'))-b(V_{g,s}(t,z'))\bigr)\, dtdz'\nonumber\\
&+\sum_{l=n}^\infty \int_{\R}\int_{t_n^k}^{t_n^{k+1}} \Delta p_{t_n^{k+1}-t}(z_1-z',z_2-z') \nnn\times\\
&\hphantom{+\sum_{l=n}^\infty \int_{\R}\int_{t_n^k}^{t_n^{k+1}}}\times\bigl(b(V_{g,s}(t,z')+\psi_{l+1}(t,z'))-b(V_{g,s}(t,z')+\psi_{l}(t,z'))\bigr)\, dtdz'\nnn\\
=&:I_{21}+\sum_{l=n}^\infty I_{22}(l).
\end{align}
Thus, continuity Lemma~\ref{L:cont_gen} allowed us to pass from continuous function $\psi$ to its piecewise-constant approximations $\psi_l$. This is crucial due to the fact that our main tool, Lemma~\ref{C:31}, works only for constant in $t$ functions $x$ and $y$.

Estimation of $I_{21}$ is straightforward. It follows from the definition of approximation $\psi_n$, that
$\psi_n(t,z)=\psi_n(t^{k+1}_n,z)$ for any $t\in(t_n^k,t_n^{k+1}]$, $z\in\R$. Taking into account Remark~\ref{R:55}, we make use of the bounds in \eqref{mestfunc2} and \eqref{smartbound} to obtain
\begin{align}\label{Lipshsecterm}
|I_{21}|&\le C_\varphi |z_1-z_2| 2^{-(1/6-\delta) n}\|\psi(t_n^{k+1},\cdot)\|_w\Lambda(|z_1|\vee |z_2|\vee1)\nnn\\
&\le C  C_\varphi|z_1-z_2| 2^{-(1/6-\delta) n}\Lambda(|z_1|\vee |z_2|\vee1)(\|\psi(0,\cdot)\|_w+\alpha 2^{-m/2}).
\end{align}

Now let's do the most tricky part and work with term $I_{22}(l)$ in \eqref{Lipshfirstappr2}. As mentioned above, the functions  $\psi_{l}$ and
$\psi_{l+1}$ are piecewise-constant. Thus, we split the integral over $(t_n^k, t_n^{k+1}]$ into $2^{l-n}$ integrals over smaller subintervals
$(t_l^i, t_l^{i+1}]$, $k2^{l-n}\le i < (k+1)2^{l-n}$. On each such subinterval function $\psi_{l}$ is constant in $t$ and is equal to $\psi(t^{i+1}_l,z)$. The function $\psi_{l+1}$ also equal to the same value $\psi(t^{i+1}_l,z)$ for $t\in (t^{i+1/2}_l, t^{i+1}_l]$ and to
$\psi(t^{i+1/2}_l,z)$ for $t\in (t^{i}_l, t^{i+1/2}_l]$. Thus, we get
\begin{align}\label{razbivka}
I_{22}(l)= \sum_{i=k2^{l-n}}^{(k+1)2^{l-n}-1} &\int_{\R}\int_{t^i_l}^{t^{i+1}_l} \Delta p_{t_n^{k+1}-t}(z_1-z',z_2-z')\times\nnn\\ \times&\bigl(b(V_{g,s}(t,z')+\psi_{l+1}(t,z'))-b(V_{g,s}(t,z')+\psi_{l}(t,z'))\bigr)\, dtdz'\nnn\\
=\sum_{i=k2^{l-n}}^{(k+1)2^{l-n}-1} &\int_{\R}\int_{t^i_l}^{t^{i+1/2}_l} \Delta p_{t_n^{k+1}-t}(z_1-z',z_2-z')\times\nnn\\ \times&\bigl(b(V_{g,s}(t,z')+\psi(t^{i+1/2}_l,z'))-b(V_{g,s}(t,z')+\psi(t^{i+1}_l,z))\bigr)\, dtdz'.
\end{align}
We apply our main estimate \eqref{mestfunc2} and assumption \eqref{mainassump22} to this identity. We derive
\begin{align*}
|I_{22}&(l)|\\
\!\le&  C_\varphi |z_1\!-\!z_2| \Lambda(|z_1|\!\vee\! |z_2|\!\vee\!1) 2^{-l(2/3-\delta)} \sum_{i=k2^{l-n}}^{(k+1)2^{l-n}-1}(t_n^{k+1}-t^i_l)^{-1/2} \|\psi(t^{i+1}_l,\cdot)-\psi(t^{i+1/2}_l,\cdot) \|_w
\nonumber\\
\le& C_{\varphi}|z_1-z_2|\Lambda(|z_1|\vee |z_2|\vee1) \alpha   2^{-l(2/3-\delta)}\sum_{i=k2^{l-n}}^{(k+1)2^{l-n}-1} ((k+1)2^{l-n}-i)^{-1/2}
\nonumber\\
\le& C C_\varphi|z_1-z_2|\Lambda(|z_1|\vee |z_2|\vee1)\alpha  2^{-l(1/6-\delta)}2^{-n/2}
\end{align*}
and hence
\begin{equation}\label{I22est}
\sum_{l=n}^\infty I_{22}(l)\le C C_\varphi \alpha |z_1-z_2|\Lambda(|z_1|\vee |z_2|\vee1) 2^{-n(2/3-\delta)}.
\end{equation}
Combining \eqref{Lipshfirstappr2}, \eqref{Lipshfirstterm}, \eqref{I2split}, \eqref{Lipshsecterm}, and \eqref{I22est}, we finally obtain
\begin{align*}
Lip_n^{k+1}&\le e^{2^{-n}}(1+C 2^{-n\delta/2})Lip_n^k+ C C_\varphi 2^{-n (1/6-\delta)}(\|\psi(0,\cdot)\|_w+\alpha 2^{-m/2})+C C_\varphi\alpha 2^{-n(2/3-\delta)}\\
&=:a_n Lip_n^k+ b_n.
\end{align*}
We consider again binary approximations to $t^k_n$ and employ again the same argument as we used in the proof of \eqref{smartbound}. We get
\begin{equation*}
Lip_n^{k+1}\le \left(\prod_{i=m+1}^{n} a_i \right) \Bigl(Lip^0_m + \sum_{i=m+1}^{n} b_i\Bigr)
\end{equation*}
Since
\begin{equation*}
\prod_{i=m+1}^{n} a_i =\prod_{i=m+1}^{n}  e^{2^{-i}}(1+C2^{-i\delta/2})\le\exp(\sum_{i=m+1}^{n}  (2^{-i}+C 2^{-i\delta/2}))
\end{equation*}
is bounded uniformly in $n$, we derive
\begin{equation*}
Lip_n^{k}\le C  C_\varphi \|\psi(0,\cdot)\|_{1,\delta}+C\alpha 2^{-m(2/3-\delta)},\quad n\ge m,\,\,k\in[0,2^{n-m}-1].
\end{equation*}
Using again continuity of the function $\psi$ and the fact that the constants $C$, $C_\varphi$ in the bound above do not
depend on $k$, $n$, we arrive to the bound
\begin{equation}\label{smartboundlip}
Lip(\psi(t,\cdot))\le C  C_\varphi \|\psi(0,\cdot)\|_{1,\delta}+C\alpha 2^{-m(2/3-\delta)},\quad t\in[0,2^{-m}].
\end{equation}

\bigskip
Now we return to the main line of the proof. Recall that our aim is to estimate left-hand side of  \eqref{mainassump22} and bound $\alpha$. We treat  large $|z|$ and small $|z|$ differently. Therefore, we fix  a large threshold $M>1$.
The precise value of $M$ will depend on $m$ and will be chosen later.

If $|z|$ is large enough, we use very rough estimates from Lemma~\ref{L:30}:
\begin{equation}\label{bigR}
\sup_{|z|\ge M}e^{-|z|}|(\psi(t^{k+1}_n,z)-\psi(t^k_n,z))|\le C_{\varphi} e^{-M}2^{-n/2}.
\end{equation}

For small $z$ we estimate the same quantity more precisely. Arguing similarly to \eqref{I2split} and using \eqref{newpsi2} and continuity Lemma~\ref{L:cont_gen}, we obtain for any $z\in\R$, and integers $n\ge m$, $k\in[0,2^{n-p}-1]$
\begin{align}\label{firstappr22}
\psi&(t^{k+1}_n,z)-\psi(t^k_n,z)=\int_{\R} p_{2^{-n}}(z')(\psi(t^k_n,z-z')-\psi(t^k_n,z))\,dz'\nonumber\\
&+\int_{\R}\int_{t^k_n}^{t^{k+1}_n} p_{t^{k+1}_n-t}(z-z')\bigl(b(V_{g,s}(t,z')+\psi_n(t,z'))-b(V_{g,s}(t,z'))\bigr)\, dtdz'\nonumber\\
&+\sum_{l=n}^\infty \int_{\R}\int_{t^k_n}^{t^{k+1}_n} p_{t^{k+1}_n-t}(z-z') \bigl(b(V_{g,s}(t,z)+\psi_{l+1}(t,z))-b(V_{g,s}(t,z)+\psi_{l}(t,z))\bigr)\, dtdz'\nnn\\
=:&J_1(k,z)+J_2(k,z)+J_3(k,z)
\end{align}

By the definition of $Lip^k_n$ and \eqref{smartboundlip} we have
\begin{equation}\label{firstbs}
\sup_{|z|\le M} e^{-|z|}|J_1(k,z)|\le  C Lip^k_n M^\delta 2^{-n/2}\le C C_\varphi M^{\delta}2^{-n/2}( \|\psi(0,\cdot)\|_{1,\delta}+\alpha 2^{-m(2/3-\delta)}).
\end{equation}

To bound $e^{-|z|}J_2(k,z)$ we apply estimate \eqref{mestfunc1}. We get
\begin{align*}
\sup_{|z|\le M}& e^{-|z|}|J_2(k,z)|\\
=&\sup_{|z|\le M}e^{-|z|}\Bigl|\int_{\R}\int_{t^k_n}^{t^{k+1}_n} p_{t^{k+1}_n-t}(z-z')\bigl(b(V_{g,s}(t,z')+\psi(t^{k+1}_n,z'))-b(V_{g,s}(t,z'))\bigr)\, dtdz'\Bigr|\nonumber\\
\le& C_\varphi M^\delta \|\psi(t^{k+1}_n,\cdot)\|_w 2^{- n(2/3-\delta)}.
\end{align*}
Hence, by \eqref{smartbound} we have
\begin{align}\label{secondbs}
\sup_{|z|\le M} e^{-|z|}|J_2(k,z)|&\le
C C_\varphi  2^{-n(2/3-\delta)}M^\delta (\|\psi(0,\cdot)\|_w+\alpha 2^{-m/2}).
\end{align}

Finally, let's work with $J_3$. We estimate it using \eqref{mestfunc1} and \eqref{mainassump22}. We derive (similarly to \eqref{razbivka})
\begin{align}\label{diffDaviel22}
\sup_{|z|\le M} e^{-|z|}|J_3(k,z)|&\le C_\varphi  M^\delta\sum_{l=n}^\infty \sum_{i=k2^{l-n}}^{(k+1)2^{l-n}-1}  2^{-l(2/3-\delta)} \|\psi(t^{i+1}_l,\cdot)-\psi(t^{i+1/2}_l,\cdot) \|_w\nonumber\\
&\le C_\varphi  M^\delta\alpha\sum_{l=n}^\infty  2^{-n}2^{-l(1/6-\delta)}\nnn\\
&\le C C_\varphi  M^\delta\alpha 2^{-n(7/6-\delta)}.
\end{align}

Thus, it remains just to combine the obtained estimates of terms in the right-hand side of \eqref{firstappr22} (namely, we combine \eqref{bigR}, \eqref{firstbs}, \eqref{secondbs}, and \eqref{diffDaviel22}) to obtain for any integers $n\ge m$, $k\in[0,2^{n-m}-1]$
\begin{equation*}
\|\psi(t^{k+1}_n,\cdot)-\psi(t^k_n,\cdot)\|_w\le C C_\varphi 2^{-n/2}  M^\delta(\|\psi(0,\cdot)\|_{1,\delta}+\alpha2^{-m(2/3-\delta)})+C_\varphi  e^{-M}2^{-n/2}.
\end{equation*}

Comparing this with \eqref{mainassump22} and using the definition of $\alpha$, we get
\begin{equation*}
\alpha\le C C_\varphi  M^\delta(\|\psi(0,\cdot)\|_{1,\delta}+\alpha2^{-m(2/3-\delta)})+C_\varphi  e^{-M}.
\end{equation*}
Now we pick $M=2^{m/(2\delta)}$ and rewrite the obtained bound
\begin{equation*}
\alpha\le C C_\varphi2^{- m(1/6-\delta)}\alpha +   C C_\varphi 2^{m/2}\|\psi(0,\cdot)\|_{1,\delta}+C_\varphi e^{-2^{m/(2\delta)}}.
\end{equation*}
Recall that the constants $C$, $C_\varphi$ do not depend on $m$. Thus, if we  choose now $m_0=m_0(\delta,\varphi)$ large enough such that  $C C_\varphi2^{- m_0(1/6-\delta)}<1/2$, then for any $m\ge m_0$ we finally deduce
\begin{equation}\label{alphabound}
\alpha\le C C_\varphi 2^{m/2}\|\psi(0,\cdot)\|_{1,\delta}+C C_\varphi e^{-2^{m/(2\delta)}}.
\end{equation}
Let us stress once again, that the constants that appear in the right-hand side of the above equation ($C$, $C_\varphi$, $m_0$) depend only on $\delta$ and $\varphi$, but not on $r$ ($r$ was assumed to be equal to $0$ in the beginning of the proof). The proof for other values of $r$
is exactly the same with exactly the same final constants $C$, $C_\varphi$, $m_0$.

To complete the proof it remains just to substitute the obtained estimate of $\alpha$ \eqref{alphabound} into estimates \eqref{smartbound} and \eqref{smartboundlip}. Thus, we get the following final bound on the H\"older norm of $\psi(t,\cdot)$
\begin{equation*}
\|\psi(t,\cdot)\|_{1,\delta}=\|\psi(t,\cdot)\|_w+Lip_\delta(\psi(t,\cdot))\le
C C_\varphi \|\psi(0,\cdot)\|_{1,\delta}+C C_\varphi e^{-2^{m/(2\delta)}},\quad t\in[0,2^{-m}].
\end{equation*}
This proves \eqref{diffonestepres}. Estimate \eqref{diffonestepresuseful} is an immediate corollary
of  \eqref{diffonestepres}.
\end{proof}

Now we can straightforwardly estimate the behavior of $\psi$ on bigger intervals and give a proof of
uniqueness part of Theorem~\ref{T:exandpbp}.
\begin{proof}[Proof of uniqueness part of Theorem~\ref{T:exandpbp}]
We use Lemma~\ref{L:212} with $\delta=1/10$ and zero initial condition $\varphi=0$. For large enough
$m\ge m_0$ we apply bound \eqref{diffonestepresuseful} consequently  $2^m$ times. We get that there exists a constant $C>0$ such that for any integer
$k\in[0,2^{-m}]$
\begin{equation*}
\|\psi(k2^{-m},\cdot)\|_{1,\delta}\le C^{2^m}\exp(-2^{m/(2\delta)}).
\end{equation*}
Note that the constant $C$ does not depend on $m$ or $k$. Therefore, by  letting $m\to\infty$ we get
$\|\psi(t,\cdot)\|_{1,\delta}=0$ for any dyadic $t\in[0,1]$. By continuity of $\psi$, we see that
$\|\psi(t,\cdot)\|_{1,\delta}=0$ for any $t\in[0,1]$, which implies that $\psi$ is identically $0$ on $[0,1]\times\R$.

Thus, equation \eqref{newpsi2} has only trivial solution and therefore on $\Omega'$ equation \eqref{SPDEmildflow}
has a unique solution. This solution is in $\mathbf B(0+)$ by Proposition~\ref{P:51}.

Finally, let us prove the last part of the theorem. Let $q_1,q_2\in\mathbf B(0+)$ be two initial conditions and $q_1(z)=q_2(z)$ Lebesgue--almost everywhere. Let $u_{s,q_1}$ and $u_{s,q_2}$ be the solutions to \eqref{SPDEmildflow} that start
with initial conditions $q_1$ and $q_2$ correspondingly. Note that for $t\in(0,1]$, $z\in\R$  we have
\begin{equation}\label{difftwo}
 \int_\R p_t (z-z')q_1(z')\,dz'=\int_\R p_t (z-z')q_2(z')\,dz'.
\end{equation}
Consider now the function $v(t,z):=u_{s,q_2}(t,z,\omega)$ for $t\in(0,1]$, $z\in\R$, and $v(0,z):=q_1(z)$, $z\in\R$. Identity \eqref{difftwo}
implies that $v$ is another solution to \eqref{SPDEmildflow} that start
with initial condition $q_1$ at time $s$. By uniqueness, $v=u_{s,q_1}$. Thus,
$u_{s,q_1}(t,z)=u_{s,q_2}(t,z)$ for $t\in(0,1]$, $z\in\R$.
\end{proof}

\subsection{Proof of Theorem~\ref{T:stochasticflow}(b): continuity of the flow}\label{S:cont}
\begin{proof}[Proof of Theorem~\ref{T:stochasticflow}(b)]
Fix $\omega\in\Omega'$, $s\ge0$, $t> s$, $z\in\R$ and the initial conditions $(q_n)_{n\in\Z_+}$ satisfying the assumptions of the theorem. It follows from the definition of
$\phi$ in Part (a) of the theorem that  $\phi(s,\cdot,q_n,\omega)$ is a solution to \eqref{SPDEmildflow} that starts at time $s$ with the condition $q_n$. Since all the assumptions of Lemma~\ref{L:contincond} are met, we see that there exist a subsequence $(n_k)_{k\in\Z_+}$
such that
\begin{equation*}
\lim_{k\to\infty}\phi(s,t,q_{n_k},\omega)(z)= \widetilde u (t,z,\omega)
\end{equation*}
for some  $\widetilde u$ and $\widetilde u (\cdot,\cdot,\omega)$  solves  \eqref{SPDEmildflow} that starts at time $s$ with the initial condition $q$. On the other hand, $\phi(s,\cdot,q,\omega)$ is also a solution to
\eqref{SPDEmildflow} that starts at time $s$ with the initial condition $q$. Therefore, by Theorem~\ref{T:exandpbp}, these solutions coincide and $\widetilde u (t,z,\omega)=\phi(s,t,q,\omega)(z)$.

Thus, $(\phi(s,t,q_{n},\omega)(z))_{n\in\Z_+}$ is a sequence of real numbers such that each subsequence of it has a converging sub-subsequence. Since all the limiting points coincide (and equal to $\phi(s,t,q,\omega)(z)$), we see by a standard argument that
\begin{equation*}
\lim_{n\to\infty}\phi(s,t,q_n,\omega)(z)= \phi(s,t,q,\omega)(z).\qedhere
\end{equation*}
\end{proof}

\section{Proofs of Lemma~\ref{L:VietR} (Kolmogorov  theorem)}\label{proofKol}

To give a proof of the ``global'' version of the Kolmogorov continuity theorem we need the following ``local'' version. Recall the definition of distance $d_a$ given in \eqref{distdef}.
\begin{Lemma}\label{L:Viet}
Let $X(x,y)$, $x\in[0,1]$, $y\in[0,1]^{2}$ be a continuous random field with values in $\R$. Assume that there exist nonnegative constants $C_1$, $a=(a_1,a_2)\in(0,1]^2$, $\alpha, \beta_1, \beta_2$ such that
the inequality
\begin{equation*}
\E |X(x_1,y_1)-X(x_1,y_2)-X(x_2,y_1)+X(x_2,y_2)|^{\alpha}\le C_1 |x_1-x_2|^{\beta_1}d_a(y_1-y_2)^{\beta_2}
\end{equation*}
is satisfied for any $x_1,x_2\in[0,1]$, $y_1,y_2\in[0,1]^{2}$.

Then for any $\gamma_1\in(0,(\beta_1-1)/\alpha)$ and $\gamma_2\in(0,(\beta_2-1/a_1-1/a_2)/\alpha)$ there exist a set $\Omega^*\subset\Omega$ with $\P(\Omega^*)=1$ and a random variable $K$  with $\E K^{\alpha}\le C_1$ such that for any $\omega\in\Omega'$, $x_1,x_2\in[0,1]$, $y_1,y_2\in[0,1]^{2}$ we have
\begin{equation*}
|X(x_1,y_1)-X(x_1,y_2)-X(x_2,y_1)+X(x_2,y_2)|\le  C_2 K(\omega) |x_1-x_2|^{\gamma_1}d_a(y_1-y_2)^{\gamma_2}
\end{equation*}
and the constant $C_2>0$ depends only on $a$, $\alpha$, $\beta_i$, $\gamma_i$ but not on $C_1$ or the field $X$.
\end{Lemma}
\begin{proof}
The lemma is essentially \cite[Theorem~1.4.4]{Kuni}. See also the proofs of \cite[Theorem~1.4.1]{Kuni}, \cite[Theorem~3.1]{Hule} and \cite[Corollary~1.2]{Walsh}.
\end{proof}

\begin{proof}[Proof of Lemma~\ref{L:VietR}]
Fix $\gamma_1\in(0,(\beta_1-1)/\alpha)$ and $\gamma_2\in(0,(\beta_2-1/a_1-1/a_2)/\alpha)$. We divide the spaces $\R$ and $\R^2$ into unit cubes.  By Lemma~\ref{L:Viet} there exists a universal constant $C_1=C_1(a, \alpha, \beta_i, \gamma_i, C)$ such that for any $m\in\Z$, $n\in\Z^{2}$,  $x_1,x_2\in [m,m+1]$, $y_1,y_2\in n+[0,1]^2$ we have on some set of full measure $\Omega^*_{n,m}\subset \Omega$
\begin{equation*}
|X(x_1,y_1)-X(x_1,y_2)-X(x_2,y_1)+X(x_2,y_2)|\le C_1 K_{m,n}(\omega) |x_1-x_2|^{\gamma_1}d_a(y_1-y_2)^{\gamma_2}
\end{equation*}
and  $\E K_{m,n}^\alpha(\omega)\le 1$. Set
\begin{equation*}
K(\omega):=\sup_{m,n} \frac{K_{m,n}(\omega)}{(|m|\vee|n|\vee1)^{3/\alpha}}.
\end{equation*}
We claim that $\E K^\alpha<\infty$. Indeed,
\begin{equation*}
\E K^\alpha=\E\sup_{m,n} \frac{K^\alpha_{m,n}}{(|m|\vee|n|\vee1)^3}\le \E \sum_{m,n} \frac{K^\alpha_{m,n}}{(|m|\vee|n|\vee1)^3}\le  \sum_{m,n} \frac{1}{(|m|\vee|n|\vee1)^3}<\infty.
\end{equation*}
Therefore the random variable $K(\omega)$ is almost surely finite. Hence there exists a set
$\Omega^*_{\infty}\subset\Omega$ with $\P(\Omega^*_{\infty})=1$ such that $K(\omega)<\infty$ for $\omega\in \Omega^*_{\infty}$. Thus, on
$\Omega^*:=\bigcap\limits_{n,m}\Omega^*_{n,m}\cap\Omega^*_{\infty}$ we have for any $x_1,x_2\in [m,m+1]$, $y_1,y_2\in n+[0,1]^2$
\begin{multline}\label{Poluch1}
|X(x_1,y_1)-X(x_1,y_2)-X(x_2,y_1)+X(x_2,y_2)|\\
\le C_2 (|m|\vee|n|\vee1)^{3/\alpha} K(\omega) |x_1-x_2|^{\gamma_1}d_a(y_1-y_2)^{\gamma_2}.
\end{multline}
This implies estimate \eqref{KolmogR}.

To establish inequality \eqref{KolmogR1term} we use the second assumption of the theorem.
Arguing exactly in the same way, we derive from \eqref{secassum} that there exists a set of full measure $\widetilde\Omega^*$, a random variable $\widetilde K(\omega)$
and a constant $\widetilde C=\widetilde C(a, \alpha, \beta_i, \gamma_i, C)$ such that $\E \widetilde K^\alpha\le 1$
and for any $m\in\Z$, $n\in\Z^2$, $x_1,x_2\in [m,m+1]$ we have on $\widetilde\Omega^*$
\begin{equation*}
|X(x_1,n)-X(x_2,n)|\le \widetilde C (|m|\vee|n|\vee1)^{3/\alpha} K(\omega) |x_1-x_2|^{\gamma_1}.
\end{equation*}
This inequality together with \eqref{Poluch1} yields that for $x_1,x_2\in [m,m+1]$, $y_1\in n+[0,1]^2$ one has
on $\widetilde\Omega^*\cap\Omega^*$
\begin{equation*}
|X(x_1,y)-X(x_2,y)|\le \widetilde C_1 (|m|\vee|n|\vee1)^{3/\alpha}\widetilde K_1(\omega) |x_1-x_2|^{\gamma_1}
\end{equation*}
for some random variable $\widetilde K_1$ such that $\E \widetilde K_1^\alpha\le 1$. This proves \eqref{KolmogR1term}.
\end{proof}

\section{Proof of Proposition~\ref{P:MB} (Moment bound)}\label{proofprop}

The proof of this proposition is long and tedious. Note that all the difficulties come from the fact that the function
$t\mapsto V(t,z)$ is not a semimartingale; if this were the case, then an application of It\^o's lemma would imply the required bound.
In our proof we were inspired by the ideas from \cite[Section~4]{CG}.

We begin with the following observation. We note that the process $(V(\cdot,z))_{z\in\R}$ is stationary. Hence for any $z_1,z_2\in\R$
\begin{equation}\label{Stav}
\Law(V(\cdot,z_1),V(\cdot,z_2))=\Law(V(\cdot,z_1-z_2),V(\cdot,0)).
\end{equation}
Therefore, it will be sufficient to establish inequality \eqref{mom1} in Proposition~\ref{P:MB} only for $z_1=z$, $z_2=0$, $z\in\R$. Since we have assumed that $|z_1-z_2|\le 1$, it is enough to prove \eqref{mom1} for $z_1=z$, $z_2=0$, $|z| \le 1$.
%

To present the proof  we need to introduce a number of new objects. We fix $T>0$, the function $b$ that appears in the statement of Proposition~\ref{P:MB} and consider the random function
\begin{equation}\label{Kkk}
H(t,z,\alpha,\beta):=b'(V(t,z)+\alpha)-b'(V(t,0)+\beta),\quad t\in[0,T],\,\, z,\alpha,\beta\in\R.
\end{equation}

Recall that it is assumed that $b$ is a bounded differentiable function with bounded derivative. Without loss of generality we suppose that
\begin{equation}\label{normless1}
\|b\|_{\infty}\le 1
\end{equation}
(otherwise we consider the function $\widetilde{b}:=b/\|b\|_{\infty}$ instead of $b$). All  constants that appear in this section
do not depend on the function $b$ (satisfying condition \eqref{normless1}).

Fix $z,\alpha,\beta\in\R$, $t_1,t_2\in[0,T]$, $t_1<t_2$ and define a martingale $M^{t_1,t_2}=(M^{t_1,t_2}_t)_{t_1\le t\le t_2}$, where
\begin{equation}\label{Mmm}
M^{t_1,t_2}_t:=\E \bigl[\int_{t_1}^{t_2} H(r,z,\alpha,\beta)\,dr\bigl|\F_t\bigr],\quad t_1\le t\le t_2.
\end{equation}
Recall that $(\F_t)_{t\ge0}$ is the filtration associated with $\dot W$. By definition, for any $z\in\R$, $0\le s\le t$ the random variable $V(s,t,z,\omega)$ is $\F_t$-measurable.

Using the new notation and taking into account formula \eqref{Stav}, we can rewrite  the left-hand side of inequality \eqref{mom1} in Proposition~\ref{P:MB} as
$\E |M_{t_2}^{t_1,t_2}|^p$. We clearly have
\begin{equation}\label{mainest}
\E |M_{t_2}^{t_1,t_2}|^p\le  C \E |M_{t_1}^{t_1,t_2}|^p+ C \E|M_{t_2}^{t_1,t_2}-M_{t_1}^{t_1,t_2}|^p.
\end{equation}
The first term in the right-hand side of \eqref{mainest} is easy to bound. Namely we first estimate $\E \bigl[H(r,z,\alpha,\beta)|\F_{t_1}\bigr]$
for $r\in[t_1,t_2]$ (this is done in Lemma~\ref{L:AL:fs}) and then apply the integral Minkowski inequality.

To estimate the second term in the right-hand side of \eqref{mainest}, we first calculate the quadratic variation of the martingale $M^{t_1,t_2}$ (that is, $[ M^{t_1,t_2}]_t$) and then apply the Burk\-holder--Davis--Gundy inequality.

Let us briefly explain how we estimate  $[ M^{t_1,t_2}]_t$. To simplify the notation, if there is no ambiguity, further we drop the superindex $(t_1,t_2)$ and write $M_t$ instead of $M^{t_1,t_2}_t$. Recall that for continuous martingales quadratic variation $[\cdot]_t$ equals predictable quadratic variation $\langle\cdot\rangle_t$. To calculate $\langle M\rangle_t$ we use the following identity valid for $t_1\le r\le s\le t_2$
\begin{align}\label{veryimpsec8}
M_s-M_r=&\int_{t_1}^{t_2} \E ([H(t,z,\alpha,\beta)\bigl|\F_s\bigr] - \E[H(t,z,\alpha,\beta)\bigl|\F_r\bigr])\,dt\nnn\\
=&\int_s^{t_2} \E ([H(t,z,\alpha,\beta)\bigl|\F_s\bigr] - \E[H(t,z,\alpha,\beta)\bigl|\F_r\bigr])\,dt\nnn\\
&+\int_r^s H(t,z,\alpha,\beta)\,dt-\int_r^s \E[H(t,z,\alpha,\beta)\bigl|\F_r\bigr]\,dt\nnn\\
=:&I_1(r,s,t_2)+I_2(r,s)-I_3(r,s).
\end{align}
Note that the terms $I_1$, $I_2$, $I_3$ depend also on $\alpha,\beta,z$. To simplify the notation we omit this dependence.

Thus, we can bound the term $\E \bigl((M_s-M_r)^2\bigr|\F_r\bigr)$, which is
the main ingredient of $\langle M\rangle_t$, as follows
\begin{equation*}
\E \bigl((M_s-M_r)^2\bigr|\F_r\bigr)\le  C \E (I_1(r,s,t_2)^2|\F_r)+C \E (I_2(r,s)^2|\F_r)+ C  I_3(r,s)^2,
\end{equation*}
where we have also used $\F_r$-measurability of $I_3(r,s)$. The corresponding bounds for the terms in right--hand side of the above inequality
are obtained in Lemmas~\ref{L:AL:fs}--\ref{L:AL:ftrm}. We combine all these bounds together in Lemma~\ref{L:AL:cont}.

To prove that the martingale $M$ is continuous we split it into two parts
$M:=L+N$, where
\begin{align}
L_t&:=I_2(t_1,t)=\int_{t_1}^t H(r,z,\alpha,\beta)\,dr,\quad t_1\le t \le t_2;\label{N}\\
N_t&:=\E \bigl[\int_{t}^{t_2} H(r,z,\alpha,\beta)\,dr\bigl|\F_t\bigr],\quad t_1\le r\le t_2.\label{Mtilde}
\end{align}
Since $H$ is bounded, the process $L$ is obviously continuous. To prove that $N$ is also continuous we employ the Kolmogorov continuity theorem. We use the identity $N_s-N_r=I_1(r,s,t_2)-I_3(r,s)$ valid for $t_1\le r\le s\le t_2$. This is done in Lemma~\ref{L:quadcov}.

Finally we calculate  quadratic variation $[M]_t$, which, by the continuity of $M$, is equal to $\langle M\rangle_t$. This is
done also in Lemma~\ref{L:quadcov}. At the end of the section we combine the obtained estimates and prove Proposition~\ref{P:MB}.

We start  with three auxiliary statements about the Gaussian density function.
\begin{Lemma}\label{L:GE:main} For any $\delta_1,\delta_2\in[0,1]$ there exists $C=C(\delta_1,\delta_2)$ such that for any $a_1, a_2\in\R$, $t>0$  we have the following bounds:
\begin{align}
\int_{\R}|&p_t(x+a_1)-p_t(x)|\,dx\le C |a_1|^{\delta_1}t^{-\delta_1/2};\label{GE0}\\[1ex]
\int_{\R}|&\frac{\d p_t}{\d x}(x+a_1)-\frac{\d p_t}{\d x}(x)|\,dx\le C |a_1|^{\delta_1}t^{-(1+\delta_1)/2};\label{GE1}\\[1ex]
\int_{\R}|&\frac{\d p_t}{\d x}(x+a_1+a_2)-\frac{\d p_t}{\d x}(x+a_1)-\frac{\d p_t}{\d x}(x+a_2)+\frac{\d p_t}{\d x}(x)|\,dx\nnn\\
&\phantom{\frac{\d p_t}{\d x}(x+a_1)-\frac{\d p_t}{\d x}(x)|\,dx}\le C |a_1|^{\delta_1}|a_2|^{\delta_2}t^{-(1+\delta_1+\delta_2)/2}\label{GE2}.
\end{align}
\end{Lemma}

Lemma~\ref{L:GE:main} is purely technical and its proof is placed in the Appendix.

\begin{Lemma}\label{L:GE:gauss}
Let $b$ be a bounded differentiable function with bounded derivative and $\|b\|_\infty\le1$. Then for any $\delta_1,\delta_2,\delta_3\in[0,1]$ there exists a constant $C=C(\delta_1,\delta_2,\delta_3)$  such that for any $a_0, a_1, a_2,a_3\in\R$ and any Gaussian random variable $X$ with zero mean and variance $\Var X= \sigma^2$  we have the following bounds:
\begin{align}
\Bigl|\E\bigl(b'(X+&a_0+a_1)-b'(X+a_0)\bigr)\Bigr|
\le C \sigma^{-1-\delta_1}|a_1|^{\delta_1};\label{GE2terms}\\[2ex]
\Bigl|\E\bigl(b'(X+&a_0+a_1+a_2)-b'(X+a_0+a_1)-b'(X+a_0+a_2)+b'(X+a_0)\bigr)\Bigr|\nnn\\
\le& C \sigma^{-1-\delta_1-\delta_2}|a_1|^{\delta_1}|a_2|^{\delta_2};\label{GE4terms}\\
\Bigl|\E\bigl(b'(X+&a_0+a_1+a_2)-b'(X+a_0+a_3)-b'(X+a_0+a_2)+b'(X+a_0)\bigr)\Bigr|\nonumber\\
\le&  C \sigma^{-1-\delta_3}|a_1-a_3|^{\delta_3}+
C \sigma^{-1-\delta_1-\delta_2}|a_1|^{\delta_1}|a_2|^{\delta_2}.\label{GEbi}
\end{align}

\end{Lemma}
\begin{proof}
First, we establish bound \eqref{GE4terms}. Fix arbitrary $\delta_1, \delta_2\in[0,1]$. We use integration by parts and rewrite the left-hand side of \eqref{GE4terms} in the following form:
\begin{align*}
\Bigl|\E\bigl(b'(&X+a_0+a_1+a_2)-b'(X+a_0+a_1)-b'(X+a_0+a_2)+b'(X+a_0)\bigr)\Bigr|
\nonumber\\
=&\Bigl|\int_{\R} \bigl(b(x+a_0+a_1+a_2)-b(x+a_0+a_2)-b(x+a_0+a_1)+b(x+a_0)\bigr)p'_{\sigma^2} (x)\,dx\Bigr|\\
=&\Bigl|\int_{\R} b(x+a_0)\bigl(p'_{\sigma^2} (x-a_1-a_2)-p'_{\sigma^2} (x-a_2)-p'_{\sigma^2} (x-a_1)+p'_{\sigma^2} (x-a_0)\bigr)\,dx\Bigr|\\
\le& C \sigma^{-1-\delta_1-\delta_2}|a_1|^{\delta_1}|a_2|^{\delta_2},
\end{align*}
where for the last inequality we used boundedness of the function $b$ and Lemma~\ref{L:GE:main}.
Estimate \eqref{GE2terms} is derived by the same argument. Estimate \eqref{GEbi} is a direct corollary of \eqref{GE2terms} and \eqref{GE4terms}.
\end{proof}

Recall the definition of $\V$ in \eqref{functI}.

\begin{Lemma}\label{L:GE:WN}
For any $\delta\in[0,1]$ there exists a constant $C=C(\delta)>0$ such that for any $0\le r<s\le t$, $z,z_1,z_2\in\R$ we have
\begin{align}
&\Var \V(r,s,t,z)\le C (s-r)(t-s)^{-1/2};\label{var}\\
&\Var \bigl(\V(r,s,t,z_1)-\V(r,s,t,z_2)\bigr)\le C |z_1-z_2|^\delta(s-r)(t-s)^{-1/2-\delta/2};\label{dif1}\\
&\Var \bigl(\V(r,s,t,z_1)-\V(r,s,t,z_2)\bigr)\le C |z_1-z_2|\label{dif2}.
\end{align}
\end{Lemma}
\begin{proof}
To prove the lemma we employ formula for the covariance of stochastic integrals, see, e.g., \cite[Proposition~5.18]{Khosh}. Inequality \eqref{var} is verified by a direct calculation. Inequalities \eqref{dif1} and \eqref{dif2} are established by the following argument:
\begin{align*}
\Var \bigl(\V(r,s,t,z_1)&-\V(r,s,t,z_2)\bigr)\\&=\Var \V(r,s,t,z_1)+\Var \V(r,s,t,z_2) - 2 \cov( \V(r,s,t,z_1), \V(r,s,t,z_2))\\
&= 2\int_r^s (p_{2(t-u)}(0)-p_{2(t-u)}(z_1-z_2))\,du\\
&=\int_{t-s}^{t-r}\frac{1}{\sqrt\pi\sqrt{u}} (1-e^{-\frac{(z_1-z_2)^2}{4u}})\,du\\
&\le \int_{t-s}^{t-r}\frac{1}{\sqrt\pi\sqrt{u}}\Bigl(\frac{(z_1-z_2)^2}{4u}\wedge 1\Bigr) \,du\\
&=\int_{(t-s)\wedge \frac{(z_1-z_2)^2}{4}}^{(t-r) \wedge \frac{(z_1-z_2)^2}{4}}\frac{1}{\sqrt\pi\sqrt{u}}\,du
+ \int_{(t-s)\vee \frac{(z_1-z_2)^2}{4}}^{(t-r) \vee \frac{(z_1-z_2)^2}{4}}\frac{(z_1-z_2)^2}{4\sqrt\pi u^{3/2}}.
\end{align*}
By considering three different cases (namely, $(z_1-z_2)^2/4\le t-s$; $t-s< (z_1-z_2)^2/4\le t-r$; $(z_1-z_2)^2/4>t-r$), we arrive to  \eqref{dif2}. We also get
\begin{equation*}
\Var \bigl(\V(r,s,t,z_1)-\V(r,s,t,z_2)\bigr)\le C |z_1-z_2|^\delta(t-s)^{-\delta/2}(\sqrt{t-r}-\sqrt{t-s}),
\end{equation*}
from which \eqref{dif1} follows immediately.
\end{proof}

Now we are moving on to calculating the quadratic variation of the martingale $M$. In the next three lemmas we will obtain a moment bound on $I_1$ (recall its definition in \eqref{veryimpsec8}).

\begin{Lemma}\label{L:AL:fs}
Let $\delta\in[0,1]$. There exists $C=C(T,\delta)$ such that for any $0\le r\le s\le t\le T$, $z, \alpha,\beta\in\R$ we have
\begin{align}\label{est1.3}
\Bigl|\E\bigl[&H(t,z,\alpha,\beta)|\F_s\bigr]-\E\bigl[H(t,z,\alpha,\beta)|\F_r\bigr]\Bigr|\nonumber\\
\le&C (t-s)^{-1/2}|\V(r,s,t,z)-\V(r,s,t,0)|\nonumber\\
&+C(t\!-\!s)^{-1/2-\delta/4}(|\V(0,s,t,z)\!-\!\V(0,s,t,0)|^{\delta}+\!|\alpha\!-\!\beta|^{\delta} )(\E |\V(r,s,t,0)|+\!|\V(r,s,t,0)|);\\
\Bigl|\E\bigl[&H(t,z,\alpha,\beta)|\F_s\bigr]\Bigr|
\le C (t-s)^{-(1+\delta)/4} \bigl(|\V(0,s,t,z)-\V(0,s,t,0)|^{\delta}+|\alpha-\beta|^{\delta} \bigr),\label{est1.31}
\end{align}
where $H$ is defined in \eqref{Kkk} .
\end{Lemma}
\begin{proof}
We start with the proof of inequality \eqref{est1.3}. Fix $0\le r\le s \le t \le T$, $z, \alpha,\beta\in\R$. For $i=1,2$ introduce the following random variables:
\begin{equation*}
X_i:=\V(0,r,t,z_i);\,\,Y_i:=\V(r,s,t,z_i);\,\,Z_i:=\V(s,t,t,z_i),
\end{equation*}
where $z_1:=0$, $z_2:=z$. We clearly have $V(t,z_i)=X_i+Y_i+Z_i$, $i=1,2$. Note also that the random vectors $(X_1,X_2)$, $(Y_1,Y_2)$ and $(Z_1,Z_2)$ are Gaussian and independent. Define for $x,y\in\R$
\begin{equation*}
J(x,y):=\E (b'(x+Z_1)- b'(y+Z_1)).
\end{equation*}
With this notation in hand we rewrite
\begin{align}\label{H1}
\E\bigl[H(t,z,\alpha,\beta)|\F_s\bigr]&=\E\bigl[(b'(X_1+Y_1+Z_1+\alpha)-b'(X_2+Y_2+Z_2+\beta))\bigl|X_1,X_2,Y_1,Y_2\bigr]\nnn\\
&=\E\bigl[(b'(X_1+Y_1+Z_1+\alpha)-b'(X_2+Y_2+Z_1+\beta))\bigl|X_1,X_2,Y_1,Y_2\bigr]\nnn\\
&=J(X_1+Y_1+\alpha,X_2+Y_2+\beta).
\end{align}
and
\begin{align}\label{H2}
\E\bigl[H(t,z,\alpha,\beta)|\F_r\bigr]&=\E\bigl[(b'(X_1+Y_1+Z_1+\alpha)-b'(X_2+Y_2+Z_2+\beta))\bigl|X_1,X_2\bigr]\nnn\\
&=\E\bigl[(b'(X_1+Y_1+Z_1+\alpha)-b'(X_2+Y_1+Z_1+\beta))\bigl|X_1,X_2\bigr]\nnn\\
&=\E [J(X_1+Y_1+\alpha,X_2+Y_1+\beta)|X_1,X_2],
\end{align}
where we also used the fact that $\Law (Y_1)=\Law (Y_2)$ and $\Law (Z_1)=\Law (Z_2)$.

Using again the independence of the vectors $(X_1,X_2)$, $(Y_1,Y_2)$ and $(Z_1,Z_2)$, we obtain
\begin{align}\label{est1.3vtorshag}
\Bigl| \E \bigl(J(X_1&+Y_1+\alpha,X_2+Y_2+\beta)\bigl|X_1=x_1,X_2=x_2,Y_1=y_1,Y_2=y_2\bigr)\nnn\\
&-\E\bigl(J(X_1+Y_1+\alpha,X_2+Y_1+\beta)\bigl|X_1=x_1,X_2=x_2\bigr)\Bigr|\nonumber\\
=&\Bigl|  [J(x_1+y_1+\alpha,x_2+y_2+\beta)-\E J(x_1+Y_1+\alpha,x_2+Y_1+\beta)]\Bigr|\nonumber\\
\le& \E\Bigl| \E \bigl[(J(x_1+y_1+\alpha,x_2+y_2+\beta)- J(x_1+Y_1+\alpha,x_2+Y_1+\beta))\bigl|Y_1\bigr]\Bigr|.
\end{align}

To estimate the conditional expectation in the last expression we employ Lemma~\ref{L:GE:gauss}, inequality \eqref{GEbi} with the following set of parameters: $\delta_1\leftarrow1$, $\delta_2\leftarrow \delta$, $\delta_3\leftarrow 1$, $a_0\leftarrow x_2+y+\beta$, $a_1\leftarrow y_1-y$,
$a_2\leftarrow x_1-x_2+\alpha-\beta$, and $a_3\leftarrow y_2-y$. We derive
\begin{align}\label{est1.3tretshag}
|\E&  \bigl[(J(x_1+y_1+\alpha,x_2+y_2+\beta)- J(x_1+Y_1+\alpha,x_2+Y_1+\beta))\bigl|Y_1=y\bigr]|\nonumber\\
=&\Bigl|\E\bigl(b'(x_1\!+y_1\!+\!Z_1\!+\alpha)-b'(x_2\!+y_2\!+\!Z_1\!+\beta)-b'(x_1\!+y\!+\!Z_1\!+\alpha)+b'(x_2\!+y\!+\!Z_1\!+\beta)\Bigr|\nonumber\\
\le&  C (\Var {Z_1})^{-1}|a_1-a_3|+
C (\Var Z_1)^{-1-\delta/2}|a_1||a_2|^{\delta}.
\end{align}
By  Lemma~\ref{L:GE:WN} we get $\Var{Z_1}=C\sqrt{t-s}$. Using this bound and \eqref{est1.3tretshag}, we continue \eqref{est1.3vtorshag} as follows
\begin{align*}
\E\Bigl| \E &\bigl[(J(x_1+y_1+\alpha,x_2+y_2+\beta)- J(x_1+Y_1+\alpha,x_2+Y_1+\beta))\bigl|Y_1\bigr]\Bigr|\nonumber\\
\le&  C (t-s)^{-1/2}|y_1-y_2|+  C (t-s)^{-(2+\delta)/4}(|x_1-x_2|^{\delta}+|\alpha-\beta|^{\delta})(\E |Y_1|+|y_1|).
\end{align*}
An application of \eqref{H1}, \eqref{H2}  and \eqref{est1.3vtorshag}  finishes the proof of bound \eqref{est1.3}.

Finally, to establish inequality \eqref{est1.31} we just note that for any $x,y\in\R$, $\delta\in[0,1]$ estimate \eqref{GE2terms} implies
\begin{equation*}
|J(x,y)|\le C (\Var Z_1)^{-1/2-\delta/2}|x-y|^{\delta}\le C (t-s)^{-(1+\delta/4)}|x-y|^{\delta}.
\end{equation*}
Combining this with \eqref{H1} we come to \eqref{est1.31}.
\end{proof}

The next statement can be called the conditional integral Minkowski inequality. This inequality is definitely  not new; however we were not able to find its proof in the literature. So we provide it here for the completeness of the argument.

\begin{Lemma}\label{L:integr}
Let $(\xi(t))_{t\ge0}$ be a random process. Then for any $\sigma$--field $\G\subset\F$ and $0\le a\le b$  we have
\begin{align*}
\E \Bigl[\bigl(\int_a^b \xi(t)\,dt\bigr)^2\Bigl|\G\Bigr]\le \bigl(\int_a^b (\E [\xi(t)^2|\G])^{1/2}\,dt\bigr)^2
\end{align*}
\end{Lemma}
\begin{proof}
By the Fubini theorem, we have
\begin{align*}
\E \Bigl[\bigl(\int_a^b \xi(t)\,dt\bigr)^2\Bigl|\G\Bigr]&= \int_a^b \int_a^b \E [\xi(t)\xi(s)|\G]\,dtds\\
&\le \int_a^b \int_a^b (\E [\xi(t)^2|\G])^{1/2}(\E[\xi(s)^2|\G])^{1/2}\,dtds\\
&=\bigl[\int_a^b (\E [\xi(t)^2|\G])^{1/2}\,dt\bigr]^2.\qedhere
\end{align*}
\end{proof}

\begin{Lemma}\label{L:AL:ftrm}
Let $p>0$, $\delta\in(0,1)$, $\delta'\in(0,\delta)$. Then there exist a random variable $K(\omega)$ and a constant $C=C(p,T,\delta,\delta')>0$ such that $\E K(\omega)^p\le C$ and for any $0\le r\le s\le t\le T$, $z,\alpha,\beta\in\R$,
$|z|\le 1$ we have almost surely
\begin{equation}\label{estcond}
\E [I_1(r,s,t)^2|\F_r]\le K(\omega)  (s-r) (t-s)^{1/2-\delta/2}\bigl(|z|^{\delta'}+|\alpha-\beta|^{2\delta}\bigr),
\end{equation}
where $I_1$ is defined in \eqref{veryimpsec8}.

We also have for $\delta\in(0,1/2)$
\begin{equation}\label{estcond4m}
\E (I_1(r,s,t))^4 \le C (s-r)^2 (t-s)^{1-\delta}(|z|^{2\delta}+|\alpha-\beta|^{4\delta}\bigr).
\end{equation}

\end{Lemma}
\begin{proof}
Fix $p>0$, $\delta\in(0,1)$, $\delta'\in(0,\delta)$. We employ estimate \eqref{est1.3} from Lemma~\ref{L:AL:fs} and use the corresponding estimates from Lemma~\ref{L:GE:leon} and Lemma~\ref{L:GE:WN} to obtain that there exists a random variable $K(\omega)=K(\omega,p,T,\delta,\delta')>0$ and a constant $C=C(p,T,\delta,\delta')$ such that $\E K(\omega)^p\le C$ and
\begin{multline}\label{integr}
\E\Bigl[\Bigl(\E\bigl(H(t',z,\alpha,\beta)|\F_s\bigr)-\E\bigl(H(t',z,\alpha,\beta)|\F_r\bigr)\Bigr)^2\Bigl|\F_r\Bigr]\\
\le  K(\omega) (t'-s)^{-3/2-\delta/2}(s-r)\bigl(|z|^{\delta'}+|\alpha-\beta|^{2\delta}\bigr)
\end{multline}
for any $0\le r\le s\le t'$, $z,\alpha,\beta\in\R$, $|z|\le 1$.
An application of the conditional integral Minkowski inequality (Lemma~\ref{L:integr}) to \eqref{integr}
leads to \eqref{estcond}.

Similarly, to establish bound \eqref{estcond4m} we also use estimate \eqref{est1.3} from Lemma~\ref{L:AL:fs} and  Lemma~\ref{L:GE:WN}. We get
\begin{equation*}
\E\Bigl(\E\bigl(H(t',z,\alpha,\beta)|\F_s\bigr)-\E\bigl(H(t',z,\alpha,\beta)|\F_r\bigr)\Bigr)^4
\le C (t'-s)^{-3-2\delta}(s-r)^2\bigl(|z|^{2\delta}+|\alpha-\beta|^{4\delta}\bigr),
\end{equation*}
The proof is competed by an application of the integral Minkowski inequality.
\end{proof}

\bigskip
We are almost ready to carry out the main goal of this subsection, that is calculating quadratic variation of the martingale $M^{t_1,t_2}$ (recall that this martingale is defined in \eqref{Mmm}).  As mentioned before, a remaining technical step is to prove that this martingale is continuous. Recall that to do it we have split $M=M^{t_1,t_2}$ into two parts: $M=N+L$, see their definitions in \eqref{N} and \eqref{Mtilde}.

\begin{Lemma}\label{L:AL:cont}
Let $p>0$, $\delta\in(0,1)$, $\delta'\in(0,\delta)$. Then there exist a random variable $K(\omega)>0$ and a constant $C=C(p,T,\delta,\delta')$ such that $\E K(\omega)^p\le C$ and for any $0\le t_1\le r\le s\le t_2\le T$, $z,\alpha,\beta\in\R$,  $|z|\le 1$ we have
\begin{align}
\E[(M_s-M_r)^2|\F_r]&\le K(\omega) (s-r)(t_2-r)^{1/2-\delta/2}(|z|^{\delta'}+|\alpha-\beta|^{2\delta})+8\|b'\|_{\infty}^2 (s-r)^2;\label{cond}\\
\E( N_s-N_r)^4&\le  C (s-r)^2 \bigl(|z|^{2\delta}+|\alpha-\beta|^{4\delta}+\|b'\|_{\infty}^4(s-r)^2\bigr).\label{uncond}
\end{align}
\end{Lemma}
\begin{proof}
Recall that according to our definitions (see the beginning of this section) we have
\begin{equation*}
\E \bigl((M_s-M_r)^2\bigr|\F_r\bigr)\le  C \E (I_1(r,s,t_2)^2|\F_r)+C \E (I_2(r,s)^2|\F_r)+ C  I_3(r,s)^2.
\end{equation*}
By Lemma~\ref{L:AL:ftrm},
\begin{equation*}
\E [I_1(r,s,t_2)^2|\F_r]\le K(\omega)  (s-r) (t_2-s)^{1/2-\delta/2}\bigl(|z|^{\delta'}+|\alpha-\beta|^{2\delta}\bigr).
\end{equation*}
Note that the terms $I_2$ and $I_3$ are of order $s-r$. Therefore they will not impact the quadratic variation. Thus, we estimate
them using a very rough estimate:
\begin{equation}\label{trchlen}
|I_2(r,s)|\le 2\|b'\|_{\infty}(s-r);\quad |I_3(r,s)|\le 2\|b'\|_{\infty}(s-r).
\end{equation}
Hence
\begin{equation*}
\E [I_2(r,s)^2|\F_r]+I_3(r,s)^2\le 8\|b'\|_{\infty}^2 (s-r)^2.
\end{equation*}
Thus, we have
\begin{equation*}
\E[(M_s-M_r)^2|\F_r]\le  K(\omega)(s-r)(t_2-s)^{1/2-\delta/2}(|z|^{\delta'}+|\alpha-\beta|^{2\delta})
+8\|b'\|_{\infty}^2 (s-r)^2,
\end{equation*}
from which \eqref{cond} follows immediately.

In a similar manner,
\begin{align*}
\E[(N_s- N_r)^4]&=\E (I_1(r,s,t_2)- I_3(r,s))^4\\
&\le C (\E I_1(r,s,t_2)^4+\E I_3(r,s)^4)\\
&\le  C (s-r)^2 (t_2-s)^{1-\delta}\bigl(|z|^{2\delta}+|\alpha-\beta|^{4\delta}\bigr)+
C    \|b'\|_{\infty}^4(s-r)^4,
\end{align*}
where we have used \eqref{trchlen} and estimate \eqref{estcond4m} from Lemma~\ref{L:AL:ftrm}. This implies
\eqref{uncond}.
\end{proof}

Now we have all the tools to bound the quadratic variation of  $M^{t_1,t_2}$.

\begin{Lemma}\label{L:quadcov}
For any $0\le t_1\le t_2\le T$, $z,\alpha,\beta\in\R$, $|z|\le 1$ the martingale $M^{t_1,t_2}$ defined in \eqref{Mmm} is continuous.
Moreover, for any $\delta\in(0,1)$, $\delta'\in(0,\delta)$ and any $p>0$ there exist a  random variable $K(\omega)>0$ and a constant $C=C(p,T,\delta,\delta')$ such that $\E K(\omega)^p\le C$ and
\begin{equation}\label{quadrocov}
[ M^{t_1,t_2},M^{t_1,t_2}]_{t_2}\le K(\omega)(|z|^{\delta'}+|\alpha-\beta|^{2\delta})(t_2-t_1)^{3/2-\delta/2}.
\end{equation}
\end{Lemma}
\begin{proof}
Fix $0\le t_1\le t_2\le T$, $z,\alpha,\beta\in\R$, $|z|\le 1$.
First we prove that the martingale $M^{t_1,t_2}$ is continuous. We make use of Lemma~\ref{L:AL:cont} to obtain for $r,s\in[t_1,t_2]$
\begin{equation*}
\E( N_s- N_r)^4\le C(T,\alpha,\beta,z)(\|b'\|_{\infty}^4\vee1) (s-r)^2.
\end{equation*}
Hence, by the Kolmogorov continuity theorem, the process $(N_s)_{t_1\le s \le t_2}$ is continuous. The process
process $(L_s)_{t_1\le s \le t_2}$ is also continuous since the function $H$ is bounded. Thus, the martingale $M^{t_1,t_2}$ is continuous as a sum of two continuous processes.

We move on and calculate predictable quadratic variation of the martingale $M^{t_1,t_2}$. We employ Lemma~\ref{L:AL:cont} to get
\begin{align*}
\langle M^{t_1,t_2}&,M^{t_1,t_2}\rangle_{t_2}=\lim_{n\to\infty}\sum_{k=0}^{n-1} \E \bigl[ (M^{t_1,t_2}_{t_1+(k+1)(t_2-t_1)/n}-M^{t_1,t_2}_{t_1+k(t_2-t_1)/n})^2\bigl|\F_{t_1+k(t_2-t_1)/n}\bigr]\\
\le& K(\omega)(|z|^{\delta'}+|\alpha-\beta|^{2\delta})(t_2-t_1)^{3/2-\delta/2}\lim_{n\to\infty} \frac{\sum_{k=0}^{n-1} (n-k)^{1/2-\delta/2}}{n^{3/2-\delta/2}}\\
&+8\|b'\|_{\infty}^2(t_2-t_1)^2\lim_{n\to\infty}n^{-1}\\
\le& K(\omega)(|z|^{\delta'}+|\alpha-\beta|^{2\delta})(t_2-t_1)^{3/2-\delta/2}.
\end{align*}
By above, the martingale $M^{t_1,t_2}$ is continuous. Hence its quadratic variation is equal to its predictable quadratic variation, that is $[M^{t_1,t_2},M^{t_1,t_2}]_t=\langle M^{t_1,t_2},M^{t_1,t_2}\rangle_t$. This proves \eqref{quadrocov}.
\end{proof}

Finally, we can prove Proposition~\ref{P:MB}.

\begin{proof}[Proof of Proposition~\ref{P:MB}]
Fix $p\ge1$. As we already pointed out at the beginning of this Section, it is sufficient to show \eqref{mom1} only for $z_1=z$, $z_2=0$, where $|z|\le1$. Note also that
\begin{equation*}
\E\Bigl(\int_{t_1}^{t_2} \bigl(b'(V(u,z)+\alpha)-b'(V(u,0)+\beta)\bigr)\,du\Bigr)^p=
\E |M^{t_1,t_2}_{t_2}|^p.
\end{equation*}

It follows from the Burkholder--Davis--Gundy inequality and Lemma~\ref{L:quadcov} that
\begin{align}\label{imiappl0}
\E |M^{t_1,t_2}_{t_2}-&M^{t_1,t_2}_{t_1}|^p\le C \E [ M^{t_1,t_2},M^{t_1,t_2}]_{t_2}^{p/2}\nonumber\\
&\le C(t_2-t_1)^{p(3/4-\delta/4)}(|z|^{ p\delta'/2}+|\alpha-\beta|^{p\delta}),
\end{align}
where we have also used the finiteness of the $p/2$-th moment of $K(\omega)$. By the integral Minkowski inequality,
\begin{align}\label{imiappl}
\E |M^{t_1,t_2}_{t_1}|^p&=\|M^{t_1,t_2}_{t_1}\|_p^p= \Bigl\|\int_{t_1}^{t_2}\E\bigl[ H(t,z,\alpha,\beta)\bigl|\F_{t_1}\bigr]\,dt\Bigr\|_p^p\nonumber\\
&\le \Bigl(\int_{t_1}^{t_2} \|\E\bigl[ H(t,z,\alpha,\beta)\bigl|\F_{t_1}\bigr]\|_p\,dt\Bigr)^p
\end{align}
We employ estimate \eqref{est1.31} from Lemma~\ref{L:AL:fs} and Lemma~\ref{L:GE:WN} to get
\begin{equation*}
\|\E\bigl[ H(t,z,\alpha,\beta)\bigl|\F_{t_1}\bigr]\|_p
\le C(t-t_1)^{-(1+\delta)/4} (|z|^{\delta/2}+|\alpha-\beta|^{\delta} ).
\end{equation*}
Combining this inequality with \eqref{imiappl} we obtain
\begin{equation*}
\E |M^{t_1,t_2}_{t_1}|^p\le
 C (t_2-t_1)^{p(3/4-\delta/4)}(|z|^{p\delta/2}+|\alpha-\beta|^{p\delta} ).
\end{equation*}
Inequality \eqref{mom1} follows now from this, \eqref{imiappl0} and the following simple observation:
\begin{equation*}
\E |M^{t_1,t_2}_{t_2}|^p\le C(
\E |M^{t_1,t_2}_{t_1}|^p+\E |M^{t_1,t_2}_{t_2}-M^{t_1,t_2}_{t_1}|^p).
\end{equation*}
The second part of Proposition~\ref{P:MB} (inequality \eqref{mom2}) is established along the same lines as
inequality \eqref{mom1}.
\end{proof}

\section{Proofs of Lemma~\ref{L:cont-prev} and Lemma~\ref{L:prob0}}\label{S:PCL}

\begin{proof}[Proof of Lemma~\ref{L:cont-prev}]
The proof is based on Proposition~\ref{P:MB} and the ideas from the proofs of \cite[Lemma~3.3]{Dav} and \cite[Lemma~3.4]{Shap}.
Before we begin the proof let us just note that in the case $f_1=const$, $f_2=const$ inequality \eqref{mest} is almost obvious; one should just
calculate the corresponding expected value and apply the Chebyshev inequality, see inequality \eqref{obs1} below. If $f_1$ and $f_2$ are piecewise constant functions, establishing \eqref{mest} is also not very difficult. Thus to prove  \eqref{mest} for the general case, we first
establish it for a suitable piecewise-continuous approximations of $f_1$, $f_2$ and then pass to the limit. Let us carry out this plan.

Fix $\eps>0$, $M>0$, $N\in\N$, $h>1/2$ and take sufficiently small $\delta$. Without loss of generality and to simplify the notation, we assume $T=1$, $\mu=1$; the proof for other values of $T$, $\mu$ is exactly the same.  We will choose a specific $\delta$ later. Let $U$ be any set such that $|U|\le\delta$. Let us verify that inequality \eqref{mest} holds on a large enough set for all $z\in[-N,N]$, $r\in\N$, $f_1,f_2\in \mathbf{H}(h,1,1,M)$. Note that by definition of the class $\mathbf{H}$, we have $\sup_{z\in[-N,N]} |f_i (z)|\le NM$, $i=1,2$.

The proof strategy relies on two observations. First, we note that the random variable $V(t,z)$ has a Gaussian distribution
with mean $0$ and variance $\sqrt{t/\pi}$, see \eqref{var}. Hence for any $x\in\R$, $z\in\R$, $0\le t_1< t_2\le2$ we have
\begin{align}\label{obs1}
\E \int_{t_1}^{t_2} \I_U (V(t,z)+x)dt&= \int_{t_1}^{t_2}\int_\R \I_U (y+x)p_{\sqrt{t/\pi}}(y)\,dy dt\nnn\\
&\le C_1 \bigl(\int_\R \I_U(y) dy\bigr)^{1/2}\int_{t_1}^{t_2} t^{-1/8}dt\nnn\\
&\le C_1 \sqrt{\delta} |t_2-t_1|^{7/8}.
\end{align}

We fix large integer $m>r$ and split the interval $[0,2]$ into $2^{m+1}$ smaller subintervals. Consider the event
\begin{equation*}
A_k(\eps,x,z):=\Bigl\{\int_{k2^{-m}}^{(k+1)2^{-m}} \I_U(V(t,z)+x)\,dt<\eps 2^{-m}\Bigr\}.
\end{equation*}
By the Chebyshev inequality, \eqref{obs1} implies $\P(A_k(x,\eps,z))\ge1-C_1 \sqrt{\delta} \eps^{-1}2^{m/8}$. Thus, for the event
\begin{equation*}
A(\eps):=\bigcap_{k=0}^{2^{m+1}-1}\bigcap_{i=-MN2^{4m+1}}^{MN2^{4m+1}}\bigcap_{j=-N2^{8m}}^{N2^{8m}} A_k(\eps,i2^{-4m},j2^{-8m}).
\end{equation*}
we have
\begin{equation*}
\P(A(\eps))\ge1-C_1 MN^22^{14m}\sqrt{\delta} \eps^{-1}.
\end{equation*}

Second, we fix $\rho\in(1/2,1)$ and $\theta\in(0,1)$ such that
\begin{equation}\label{posledka_urauraura_vosklvosklvoskl}
\rho(h-1/4)-1/4-\theta>\theta.
\end{equation}
Since we have assumed that $h>1/2$, we see that such $\rho$, $\theta$ exist. We derive from Proposition~\ref{P:MB} and the Kolmogorov continuity theorem   that for $x,y\in[-2MN,2MN]$, $z_1,z_2\in[-N,N]$, $0\le t_1< t_2\le2$ we have
\begin{equation*}
\bigl|\!\int_{t_1}^{t_2}\! (\I_U(V(t,z_1)+x)-\I_U(V(t,z_2)+y))\,dt\bigr|\le K(\omega)(t_2-t_1)^{3/4-\rho/4-\theta}
(|x-y|^{\rho}+|z_1-z_2|^{\rho/2}),
\end{equation*}
where $\E K(\omega)\le C_2$. Define for $\kappa>0$ an event
\begin{equation*}
B(\kappa):=\{K(\omega)\le \kappa\}.
\end{equation*}
The Chebyshev inequality implies that $\P (B(\kappa))\ge 1-C_2 \kappa^{-1}$. Therefore
\begin{equation}\label{meramn}
\P(A(\eps)\cap B(\kappa))\ge 1-C_1 MN^22^{14m}\sqrt{\delta} \eps^{-1}-C_2 \kappa^{-1}.
\end{equation}
We choose now large $\kappa$ such that $C_2 \kappa^{-1}\le \eps/2$.

It follows from the above definitions and a change of variables $t':=t+s$ in the integral 
that on event $A(\eps)\cap B(\kappa)$ we have
\begin{equation}\label{OdinEvent}
\int_{k2^{-m}}^{(k+1)2^{-m}} \I_U(V(t+s,z)+x)\,dt\le2\cdot2^{-m} (\eps+\kappa2^{-m}).
\end{equation}
for all $x\in[-2MN,2MN]$, $s\in[0,1]$, $z\in[-N,N]$, $0\le k \le 2^{m}-1$.

Now we fix $r\in\N$, $s\in[0,1]$, $z\in[-N,N]$, $f_1,f_2\in \mathbf{H}(h,1,1,M)$. Put
\begin{equation*}
f(t,z):=f_1(t,z)+\lambda_{r}(f_2)(t,z).
\end{equation*}
It follows from Fatou's lemma and the fact the set $U$ is open that
\begin{equation}\label{Fatouc}
\int_{0}^{1} \I_U(V(t+s,z)+f(t,z))\,dt\le \liminf_{n\to\infty} \int_0^1 \I_U (V(t+s,z)+\lambda_n(f)(t,z))\,dt,
\end{equation}
Hence for any $n\ge m$ we have
\begin{align}\label{estnepr}
\int_0^{1}\I_U &(V(t+s,z)+\lambda_n(f)(t,z))\,dt\le\int_0^{1} \I_U (V(t+s,z)+\lambda_m(f)(t,z))\,dt\nnn\\
&+\sum_{l=m}^{\infty} \Bigl|\int_0^1 \bigl(\I_U (V(t+s,z)+\lambda_{l+1}(f)(t,z))-
\I_U (V(t+s,z)+\lambda_l(f)(t,z))\bigr)\,dt\Bigr|.
\end{align}
The function $\lambda_m(f)$ is constant on time intervals $[k2^{-m},(k+1)2^{-m})$. Moreover, since $f_1,f_2\in\mathbf{H}(h,1,1,M)$, we see that $|\lambda_m(f)(t,z)|\le 2NM$ for $t\in[0,1]$, $z\in[-N,N]$. Therefore, \eqref{OdinEvent} yields that on $A(\eps)\cap B(\kappa)$
\begin{equation*}
\int_0^{1} \I_U (V(t+s,z)+\lambda_m(f)(t,z))\,dt\le 2\eps+\kappa2^{-m+1}.
\end{equation*}
In a similar way we estimate the second term in the right-hand side of \eqref{estnepr}. It follows from the definition of the approximation
operator $\lambda$ that $\lambda_{r}(f_2)((i+1)2^{-l})=\lambda_{r}(f_2)((i+1/2)2^{-l})$ for any $l>r$, $i=0,1,...2^l-1$. This observation, the definition of the set $B(\kappa)$ and a change of variables $t':=t+s$ in the integral imply that for $l>r$ on $A(\eps)\cap B(\kappa)$
\begin{align*}
\bigl|\int_0^1& \bigl(\I_U (V(t+s,z)+\lambda_{l+1}(f)(t,z))-\I_U (V(t+s,z)+\lambda_{l}(f)(t,z))\bigr)\,dt\bigr|\\
\le& \kappa2^{-l(3/4-\rho/4-\theta)}\sum_{i=0}^{2^l-1}
|f((i+1)2^{-l},z)-f((i+1/2)2^{-l},z)|^{\rho}\\
\le&  \kappa MN 2^{-l(3/4-5\rho/4+h\rho-\theta)}\sum_{i=0}^{2^l-1} (i+1/2)^{-\rho}\\
\le& 2(1-\rho)^{-1}\kappa MN  2^{-l(\rho(h-1/4)-1/4-\theta)}\\
\le& 2(1-\rho)^{-1}\kappa MN  2^{-l\theta},
\end{align*}
where in the last inequality we took into the account that $\rho$ and $\theta$ were chosen according to \eqref{posledka_urauraura_vosklvosklvoskl}.
Combining this with the previous
estimate and \eqref{estnepr}, we finally get on $A(\eps)\cap B(\kappa)$
\begin{equation*}
\int_0^{1}\I_U (V(t+s,z)+\lambda_n(f)(t,z))\,dt\le2\eps+\kappa2^{-m+1}+
2(1-\rho)^{-1}\kappa MN  2^{-m\theta}(1-2^{-\theta})^{-1}.
\end{equation*}
Recall that we have already chosen $\kappa$, $\rho$, $\theta$. Now we choose large $m$ such that the right-hand side of the above inequality
is less than $3 \eps$. Finally, we choose small $\delta$ such than the right-hand side of \eqref{meramn} is bigger than $1-\eps$.

Thus, we got that on the set $D:=A(\eps)\cap B(\kappa)$ we have
\begin{equation*}
\int_0^{1}\I_U (V(t+s,z)+\lambda_{n}(f)(t,z))\,dt\le 3\eps
\end{equation*}
and $\P(D)\ge1-\eps$. This together with \eqref{Fatouc} yields the statement of the lemma.
\end{proof}

\begin{proof}[Proof of Lemma~\ref{L:prob0}]
The lemma is proved by a straightforward application of Lemma~\ref{L:cont-prev}.
\end{proof}

\begin{appendices}
\renewcommand{\thesection}{\!\!.}
\renewcommand{\theequation}{A.\arabic{equation}}
\section{Proofs of Gaussian density estimates}

\begin{proof}[Proof of Lemma~\ref{L:GE:main}]
We begin with the first inequality. If $t=1$ and $\delta_1=0$, then \eqref{GE0} trivially follows from the bound
\begin{equation*}
\int_{\R}|p_1(x+a_1)-p_1(x)|\,dx\le 2\int_\R p_1(x)\,dx=2.
\end{equation*}

To prove \eqref{GE0} for $t=1$, $\delta_1=1$ we use the estimate
\begin{align*}
\int_{\R}|p_1(x+a_1)-p_1(x)|\,dx&\le \int_\R\int_x^{x+a_1} |p_1'(y)|\,dydx\le |a_1|\int_\R \max_{y\in[x,x+a_1]} |p_1'(y)|\,dx\\
&\le |a_1|(4\max_{x\in\R} |p_1'(x)|+2\int_{[2,+\infty)}p_1'(x)\, dx)\le M |a_1|
\end{align*}
for some $M>0$. These estimates imply \eqref{GE0} for all $\delta_1\in[0,1]$ and $t=1$. The general case (that is, $t>0$) follows now easily by the change of variables.

Inequalities \eqref{GE1} and \eqref{GE2} are established using a similar argument.
\end{proof}

\begin{proof}[Proof of Lemma~\ref{L:densest}]
Without loss of generality, assume that $s\le t$. Note that
\begin{equation*}
|p_t(z)-p_s(z)|=\left|\int_s^t \frac{\d p_u(z)}{\d u}\,du\right|\le \int_s^t \frac{p_u(z)}{2u}(1 +\frac{z^2}{u})\,du.
\end{equation*}
Hence,
\begin{align*}
\int_\R|p_t(z)-p_s(z)|)(|z|^\delta\vee1)\,dz&\le\int_s^t \int_\R\frac{p_u(z)}{2u}(1 +\frac{z^2}{u})(|z|^\delta\vee1)\,dzdu\\
&\le C \int_s^t (u^{-1}+u^{-1+\delta/2})\,du\\
&\le C(\log t -\log s).\qedhere
\end{align*}
\end{proof}

\begin{proof}[Proof of Lemma~\ref{L:GE:3der}]
We will use the following simple bound throughout the proof: there exists a constant $C>0$ such that
for any $a,b\in\R$ we have
\begin{equation*}
\Lambda_\delta(|a-b|\vee1)\le C e^{|a|+|b|}(|a|^\delta+|b|^\delta+1)\le
C e^{|a|+|b|}(|a|^\delta+1) (|b|^\delta+1)\le C \Lambda_\delta(|a|\vee1)\Lambda_\delta(|b|\vee1).
\end{equation*}

Estimate \eqref{firstestapp} follows easily from this bound:
\begin{align*}
\int_\R p_{t}(z-z') \Lambda_\delta(|z'|\vee 1)dz'&= \int_\R p_{t}(z') \Lambda_\delta(|z-z'|\vee 1)dz'\\
&\le C \Lambda_\delta(|z|\vee1)\int_\R p_{t}(z') e^{|z'|}(|z'|^\delta+1)dz'\\
&\le C \Lambda_\delta(|z|\vee1).
\end{align*}

To obtain \eqref{seqestapp} we employ the formula $\partial p_t(z)/\partial t= p_t(z)(z^2/t^2-1/t)/2$. We get
\begin{align*}
\int_{t_1}^{t_2} &\int_\R \bigl| \frac{\d }{\d t'}p_{t-t'}(z-z')\bigr| |t_2-t'|^{2/3-\delta}\Lambda_\delta(|z'|\vee 1)\, dz'dt'\\
&\le  C \Lambda_\delta(|z|\vee1)\int_{t_1}^{t_2} |t_2-t'|^{2/3-\delta} \int_\R \frac{p_{t-t'}(z')}{t-t'}(1+\frac{z'^2}{t-t'}) e^{|z'|}(|z'|^\delta+1)\, dz'dt'\\
&\le  C \Lambda_\delta(|z|\vee1)  \int_{t_1}^{t_2} |t_2-t'|^{-1/3-\delta} \int_\R p_1(z'')(1+z''^2) e^{T|z''|}(T^{\delta/2}|z''|^\delta+1)\, dz''dt'\\
&\le  C \Lambda_\delta(|z|\vee1)|t_2-t_1|^{2/3-\delta}.
\end{align*}

To get \eqref{thirdstapp} we observe that the function $x\to|\d p_1(x)/\d x|$ is decreasing for $x>2$ and is bounded by $1$ for $x\in[0,2]$. Hence for any $x,y\in\R$ with $|x|\le |y|$ we have
\begin{equation*}
|p_1(x)-p_1(y)|\le\int_{|x|}^{|y|} |\frac{\partial}{\d r} p_1(r)|\, dr\le |x| p_1(x)|y-x| + |y-x|\I(|x|\le 2).
\end{equation*}
Therefore, for any $x,y\in\R$, $t>0$ we get
\begin{equation*}
|p_t(x)-p_t(y)|\le t^{-1}|y-x| (|x| p_t(x)+|y| p_t(y)+ \I(|x|\le 2\sqrt t)+\I(|y|\le 2\sqrt t)).
\end{equation*}
We derive using this bound
\begin{align*}
\int_\R &|p_{t}(z_1-z')-p_{t}(z_2-z')| \Lambda_\delta(|z'|\vee 1)dz'\\
\le& t^{-1}|z_1-z_2|\int_\R (p_{t}(z_1-z')|z_1-z'|+p_{t}(z_2-z')|z_2-z'|) \Lambda_\delta(|z'|\vee 1)dz'\\
&+t^{-1}|z_1-z_2|\int_\R\I(|z_1-z'|\le 2\sqrt t)+\I(|z_2-z'|\le 2\sqrt t)) \Lambda_\delta(|z'|\vee 1)dz'\\
\le& Ct^{-1}|z_1-z_2|( \Lambda_\delta(|z_1|\vee1)+\Lambda_\delta(|z_2|\vee1))\int_\R p_{t}(z')|z'| e^{|z'|}(|z'|^\delta\vee 1)dz'\\
&+C t^{-1}|z_1-z_2|( \Lambda_\delta(|z_1|\vee1)+\Lambda_\delta(|z_2|\vee1))\int_\R \I(|z'|\le 2\sqrt t) e^{|z'|}(|z'|^\delta\vee 1)dz'\\
\le& C t^{-1/2}|z_1-z_2| \Lambda_\delta(|z_1|\vee|z_2|\vee1).
\end{align*}
This proves \eqref{thirdstapp}.

The same trick is used to obtain \eqref{fourthstapp}. The function $x\to|\d^3 p_1(x)/\d x^3|$ is decreasing for $x>4$ and is bounded by $2$ for $x\in[0,4]$. Thus, for any $x,y\in\R$ with $|x|\le |y|$ we have
\begin{equation*}
\bigl|\frac{\d^2 p_1(x)}{\d x^2}-\frac{\d^2 p_1(y)}{\d y^2}\bigr|\le\int_{|x|}^{|y|} |\frac{\partial^3}{\d r^3} p_1(r)|\, dr\le |x|^3 p_1(x)|y-x| + 2|y-x|\I(|x|\le 4).
\end{equation*}
Since $\d p_t(x)/\d t=1/2 \d^2 p_t(x)/\d x^2$,  we obtain for any $x,y\in\R$, $t>0$
\begin{multline*}
\bigl|\frac{\d}{\d t}(p_{t}(x)-p_{t}(y))\bigr|\le C t^{-3}|y-x| (|x|^3 p_t(x)+|y|^3 p_t(y))\\
+ C t^{-2}|y-x|(\I(|x|\le 4\sqrt t)+\I(|y|\le 4\sqrt t)).
\end{multline*}
Thus,
\begin{multline}\label{thus}
\int_\R\int_{t_1}^{t_2}  \bigl| \frac{\d }{\d t'}(p_{t-t'}(z_1-z')-p_{t-t'}(z_2-z'))\bigr|(t_2-t')^{2/3-\delta}\Lambda_\delta(|z'|\vee 1)\, dt'dz'\\
\le C |z_1-z_2|\int_{t_1}^{t_2}(t_2-t')^{2/3-\delta} (t-t')^{-2} g(z_1,z_2,t-t')\,dt'
\end{multline}
where we denoted
\begin{align*}
g(z_1,z_2,s):=&s^{-1}\int_\R(p_{s}(z_1-z')|z_1-z'|^3+p_{s}(z_2-z')|z_2-z'|^3)\Lambda_\delta(|z'|\vee 1) dz'\\
&+\int_\R(\I(|z_1-z'|\le 4\sqrt{s})+\I(|z_2-z'|\le  4\sqrt{s}))\Lambda_\delta(|z'|\vee 1) dz'.
\end{align*}
By a direct calculation, we get
\begin{equation*}
g(z_1,z_2,s)\le C s^{1/2}\Lambda_\delta(|z_1|\vee|z_2|\vee 1).
\end{equation*}
Substituting this into \eqref{thus}, we  deduce
\begin{multline*}
\int_\R\int_{t_1}^{t_2}  \bigl| \frac{\d }{\d t'}(p_{t-t'}(z_1-z')-p_{t-t'}(z_2-z'))\bigr|(t_2-t')^{2/3-\delta}\Lambda_\delta(|z'|\vee 1)\, dt'dz'\\
\le C |z_1-z_2|\Lambda_\delta(|z_1|\vee|z_2|\vee 1)\int_{t_1}^{t_2}(t_2-t')^{2/3-\delta} (t-t')^{-3/2}\,dt'.
\end{multline*}
The integral in the right-hand side of the above equation is estimated as follows.
\begin{align*}
\int_{t_1}^{t_2}& (t_2-t')^{2/3-\delta} (t-t')^{-3/2} dt'=\int_0^{t_2-t_1}\frac{{t'}^{2/3-\delta} }{(t'+t-t_2)^{3/2} }\,dt'\\
\le&\I(t_2-t_1>t-t_2)\int_0^{t_2-t_1}{t'}^{-5/6-\delta} \,dt'+\I(t_2-t_1\le t-t_2)(t-t_2)^{-3/2} \int_0^{t_2-t_1} {t'}^{2/3-\delta} \,dt'\\
\le&2 \I(t_2-t_1>t-t_2)(t_2-t_1)^{2/3-\delta}(t-t_1)^{-1/2}\\
&+2 \I(t_2-t_1\le t-t_2)(t-t_1)^{-1/2} (t_2-t_1)^{2/3-\delta}\\
\le& 2 (t_2-t_1)^{2/3-\delta}(t-t_1)^{-1/2}.
\end{align*}
Combining this with \eqref{thus}, we finally get \eqref{fourthstapp}.
\end{proof}

\begin{proof}[Proof of Lemma~\ref{L:contextraterm}]
Without loss of generality we suppose $t_1\le t_2$. We derive
\begin{align}\label{I1I2}
|h(t_2,z_2)-h(t_1,z_1)|\le& \int_{t_1}^{t_2}\int_{\R} p_{t_2-t'}(z_2-z')|f(t',z')|\,dz'\,dt'\nnn\\
&+\int_{0}^{t_1}\int_{\R} |p_{t_2-t'}(z_2-z')-p_{t_1-t'}(z_1-z')|\,|f(t',z')|\,dz'\,dt'\nnn\\
\le& I_1+I_2.
\end{align}
Since the function $f$ is bounded, we immediately get
\begin{equation*}
I_1\le \|f\|_{\infty}(t_2-t_1).
\end{equation*}
To estimate integral $I_2$ we employ Lemma~\ref{L:densest} and Lemma~\ref{L:GE:main}. For any $\delta>0$ we have
\begin{align*}
I_2\le& \|f\|_{\infty}\int_{0}^{t_1}\int_{\R} \bigl(|p_{t_2-t'}(z_2\!-z')-p_{t_2-t'}(z_1\!-z')|+ |p_{t_2-t'}(z_1\!-z')-p_{t_1-t'}(z_1\!-z')|\bigr)\,dz'\,dt'\\
\le& C\|f\|_{\infty}\int_{0}^{t_1} |z_2-z_1|(t_2-t')^{-1/2}\,dt'+C\|f\|_{\infty}\int_{0}^{t_1} \log (t_2-t')-\log (t_1-t')\,dt'\\
\le& C \|f\|_{\infty}  |z_2-z_1|+C \|f\|_{\infty}\bigl( t_2\log t_2-(t_2-t_1)\log(t_2-t_1)-t_1\log t_1\bigr)\\
\le& C \|f\|_{\infty}  |z_2-z_1|+C\|f\|_{\infty} t_1\log(t_2/t_1)+C\|f\|_{\infty}(t_2-t_1)\log(t_2/(t_2-t_1))\\
\le& C \|f\|_{\infty}  |z_2-z_1|+C\|f\|_{\infty} (t_2-t_1)+ C(t_2-t_1)\|f\|_{\infty}|\log(t_2-t_1)|\\
\le& C \|f\|_{\infty} (|z_2-z_1|+(t_2-t_1)^{1-\delta}).
\end{align*}
Combining this with \eqref{I1I2} and the bound on $I_1$ we come to \eqref{LipH}.
\end{proof}

\begin{proof}[Proof of Lemma~\ref{L:timeq}]
Fix $T>0$. We begin with the proof of the first statement of the lemma. Suppose that $|q(z)|\le M (|z|^\mu\vee1)$, $z\in\R$. Let us check that the function $h$ belongs to the class
$\mathbf H_T(1,1,\mu,CM)$ for some $C>0$. Let $0<t_1<t_2<T$. A direct application of Lemma~\ref{L:densest}
yields for any $z\in\R$
\begin{align*}
|h(t_2,z)-h(t_1,z)|\le&\int_{\R}|p_{t_2}(z')-p_{t_1}(z')|\,|q(z-z')|\,dz'\\
\le& M \int_{\R}|p_{t_2}(z')-p_{t_1}(z')|\,(|z|^\mu+1)(|z'|^\mu+1)\,dz'\\
\le& C M (|z|^\mu+1)(\log t_2 -\log t_1)\\
\le& C M(z^\mu\vee1)(t_2-t_1)t_1^{-1},
\end{align*}
where the constant $C$ depends only on $T$ and $\mu$. Thus, $h\in \mathbf H_T(1,1,\mu,CM)$.

Now, let us prove the second statement of the lemma. Let $q\in\mathbf{CL}$. Clearly for any $z_1,z_2\in\R$, $t\in[0,T]$ we have
\begin{equation*}
|h(t,z_1)-h(t,z_2)|\le  \int_{\R}p_{t}(z')|q(z_2-z')-q(z_1-z')|\,dz'\le C_1 |z_1-z_2|,
\end{equation*}
where $C_1=C_1(q)$. Using this inequality, we derive for  $0<t_1<t_2<T$, $z\in\R$
\begin{align*}
|h(t_1,z)-h(t_2,z)|&\le \int_\R p_{t_2-t_1}(z')|h(t_1,z-z')-h(t_1,z)|\,dz'\le
C_1 \int_\R p_{t_2-t_1}(z')|z'|\,dz'\\
\le C_1 |t_1-t_2|^{1/2}.
\end{align*}
Hence for any $0<t_1<t_2<T$, $z_1,z_2\in\R$ we get
\begin{align*}
|h(t_2,z_2)-h(t_1,z_1)|&\le |h(t_2,z_2)-h(t_2,z_1)|+|h(t_2,z_1)-h(t_1,z_1)|\\
&\le C_1 |z_1-z_2|+C_1 |t_1-t_2|^{1/2}.\qedhere
\end{align*}
\end{proof}

\end{appendices}

\end{document}